\documentclass[reqno]{amsart}

\makeatletter
\def\@tocline#1#2#3#4#5#6#7{\relax
  \ifnum #1>\c@tocdepth %
  \else
    \par \addpenalty\@secpenalty\addvspace{#2}%
    \begingroup \hyphenpenalty\@M
    \@ifempty{#4}{%
      \@tempdima\csname r@tocindent\number#1\endcsname\relax
    }{%
      \@tempdima#4\relax
    }%
    \parindent\z@ \leftskip#3\relax \advance\leftskip\@tempdima\relax
    \rightskip\@pnumwidth plus4em \parfillskip-\@pnumwidth
    #5\leavevmode\hskip-\@tempdima
      \ifcase #1
       \or\or \hskip 1em \or \hskip 2em \else \hskip 3em \fi%
      #6\nobreak\relax
    \hfill\hbox to\@pnumwidth{\@tocpagenum{#7}}\par%
    \nobreak
    \endgroup
  \fi}
\makeatother

\usepackage{amsmath}
\usepackage{amsfonts}
\usepackage{amssymb}
\usepackage{amsthm}
\usepackage{enumitem}
\usepackage{bm}
\usepackage{hyperref}
\usepackage{graphicx}
\usepackage{tikz}
\usetikzlibrary{positioning}%
\usetikzlibrary{arrows}%
\usetikzlibrary{calc}
\usepackage{tkz-berge}
\usepackage{tkz-graph}
\usepackage{tikz, tikz-cd}
\usepackage{cleveref}
\usepackage{mathtools}
\usepackage{rotating}
\usepackage[all,cmtip]{xy}
\usepackage{xcolor,colortbl}

\usepackage{xcolor}

\newtheorem{theorem}{Theorem}
\newtheorem{prop}[theorem]{Proposition}

\newtheorem{corollary}[theorem]{Corollary}
\newtheorem{obs}[theorem]{Observation}
\newtheorem{lemma}[theorem]{Lemma}

\newtheorem{claim}[theorem]{Claim}
\newtheorem{maintheorem}[theorem]{Main Theorem}

\theoremstyle{definition}
\newtheorem{definition}[theorem]{Definition}
\newtheorem{example}[theorem]{Example}
\newtheorem{convention}[theorem]{Convention}
\newtheorem*{convention*}{Convention}

\newtheoremstyle{underline}%
{-1.5mm}        %
{}              %
{}              %
{\parindent}    %
{}              %
{:}             %
{1.5mm}         %
{{\underline{\thmname{#1}\thmnumber{ #2}\thmnote{(#3)}}}}  %

\theoremstyle{underline}
\newtheorem{case}{Case}

\theoremstyle{remark}
\newtheorem{remark}[theorem]{Remark}
\numberwithin{theorem}{section}
\numberwithin{equation}{section}
\newtheorem{step}{Step}

\DeclareMathOperator{\Span}{span}

\DeclareMathOperator{\Char}{Char}

\DeclareMathOperator{\stab}{stab}

\def\CC{\mathbb{C}}

\def\RR{\mathbb{R}}

\def\PP{\mathbb{P}}

\def\KK{\mathbb{K}}

\def\OO{\mathbb{O}}
\def\HH{\mathbb{H}}
\def\Z+{\mathbb{Z}_{\geq 0}}
\def\R+{\mathbb{R}_{\geq 0}}

\newcommand{\blank}{{{}\cdot{}}}
\def\iso{\operatorname{iso}}

\makeatletter
\renewcommand*\env@matrix[1][\arraystretch]{%
  \edef\arraystretch{#1}%
  \hskip -\arraycolsep
  \let\@ifnextchar\new@ifnextchar
  \array{*\c@MaxMatrixCols c}}
\makeatother

\newcommand{\rvline}{\hspace*{-\arraycolsep}\vline\hspace*{-\arraycolsep}}

\newcommand\restr[2]{{%
		\left.\kern-\nulldelimiterspace %
		#1 %
		\vphantom{\big|} %
		\right|_{#2} %
}}

\title{Universal equations for maximal isotropic Grassmannians}
\author{Tim Seynnaeve, Nafie Tairi} 

\begin{document}
\maketitle
\begin{abstract}
The isotropic Grassmannian parametrizes isotropic subspaces of a vector space equipped with a quadratic form. 
In this paper, we show that any %
maximal isotropic Grassmannian in its Pl\"ucker embedding can be defined by pulling back the equations of $Gr_{\iso}(3,7)$ or $Gr_{\iso}(4,8)$.
\end{abstract}

\section{Introduction}

The Grassmannian $Gr(k,V)$ is the space of $k$-dimensional subspaces of a vector space $V$. As an algebraic variety it is naturally embedded into the projective space $\PP(\bigwedge^kV)$ via the Pl\"ucker embedding. The defining equations of this variety are quadrics known as the Pl\"ucker relations. The smallest nontrivial example is the Grassmannian $Gr(2,4)$, which has a single defining equation known as the Klein quadric. For an arbitrary Grassmannian $Gr(k,V)$, the Pl\"ucker relations are still quadrics, but we need many quadrics of high rank 
to generate the defining ideal. However, in \cite{Kasman}, Kasman \textit{et al.}\ showed that if one only considers the Pl\"ucker relations that can be obtained from pulling back the Klein quadric, their vanishing locus is already the Grassmannian. In particular, since the Klein quadric has rank six, this implies that any Grassmannian can be set-theoretically defined by quadrics of rank six. 

If a vector space is equipped with a nondegenerate quadratic form, the variety of isotropic linear subspaces is known as the isotropic Grassmannian. 
It is natural to ask if any isotropic Grassmannian can also be set-theoretically defined by pulling back the defining equations of a fixed small isotropic Grassmannian. In this paper we show that this is indeed the case when considering \emph{maximal} isotropic subspaces in $V$. 

\begin{theorem} \label{thm:mainIntroduction}
Let $\KK$ be an algebraically closed field of characteristic not equal to $2$, and $V$ a $\KK$-vector space equipped with a nondegenerate quadratic form. %
Then the isotropic Grassmannian $Gr_{\iso}(\lfloor \frac{\dim V}{2} \rfloor,V)$ in its Pl\"ucker embedding can be defined set-theoretically by pulling back the defining equations of 
    \begin{itemize}
    \item $Gr_{\iso}(3,7)$ if $V$ is odd-dimensional, %
    \item $Gr_{\iso}(4,8)$ if $V$ is even-dimensional. %
    \end{itemize}
\end{theorem}

Since the ideals of $Gr_{\iso}(3,7)$ and $Gr_{\iso}(4,8)$, and indeed of any isotropic Grassmannian, are generated by finitely many quadrics, this implies a universal bound on the ranks of the quadrics needed to set-theoretically define any maximal isotropic Grassmannian. Notably, this bound is precisely four.

In \cite{Sato}, Draisma and Eggermont introduced the notion of a \emph{Pl\"ucker variety}, of which the Grassmannian is a special case. They proved a universality result that generalizes the universality of the Klein quadric, which was recently further generalized by Nekrasov \cite{Nekrasov}.
Our main theorem is a first step towards defining \emph{isotropic Pl\"ucker varieties}, which should generalize the isotropic Grassmannian in the same way that Pl\"ucker varieties generalize Grassmannians.

\subsection*{Organization}
The remainder of this article is organized as follows. \Cref{sec:Gras} is about ordinary Grassmannians. In this section we recall the definition of the Grassmann cone and Grassmann cone preserving maps, and state the universality of the Klein quadric by Kasman \textit{et al.} (\Cref{thm:Kasman}). In \Cref{Sec:hyperbolic spaces} we provide the necessary background on quadratic forms before we introduce isotropic Grassmannians and isotropic Grassmann cone preserving maps (IGCP maps). 
In \Cref{sec:MainResult} we state and prove the \Cref{thm:mainTheorem}, which states that if $V$  has dimension at least $9$, then starting with a $\lfloor \frac{\dim V}{2} \rfloor$-form $\omega$ not in the isotropic Grassmann cone, there is an IGCP map mapping $\omega$ outside of the isotropic Grassmann cone. This immediately implies \Cref{thm:mainIntroduction}. A key ingredient in our proof is \Cref{thm:cases}, which gives a characterization of forms in $\hat{Gr}_{\iso}(\lfloor \frac{\dim V}{2} \rfloor,V)$. %
In \Cref{sec:counterexamples} we show how the \Cref{thm:mainTheorem} fails for $\dim V$ at most $8$, by giving counterexamples for $Gr_{\iso}(3,7)$ and $Gr_{\iso}(4,8)$. The first counterexample is related to the exceptional group $G_2$ and to the octonions. Finally, in \Cref{appendix:RankOfQuadrics}, we provide a computational proof that the ideals of $Gr_{\iso}(3,7)$ and $Gr_{\iso}(4,8)$ are generated by quadrics of rank at most four. Together with \Cref{thm:mainIntroduction} this implies the same bound for any isotropic Grassmannian (\Cref{rank4}). We also sketch an alternative proof of \Cref{rank4}, which relies on a connection to the Cartan embedding instead of on \Cref{thm:mainIntroduction}.

\subsection*{Acknowledgments} 
We would like to thank Jan Draisma and Rob Eggermont for introducing us to this topic and for many helpful discussions, especially concerning the counterexamples.

\section{The ordinary Grassmannian} \label{sec:Gras}

Let $V$ be a finite-dimensional vector space over any field $\KK$, and $k \leq \dim V$. The \emph{Grassmann cone} is defined as 
\[
    \hat{Gr}(k,V) := \{v_1 \wedge \cdots \wedge v_k \mid v_1,\ldots,v_k \in V\} \subset \bigwedge \nolimits ^kV,
\]
where $\bigwedge^kV$ is the $k$'th exterior power of $V$.
For $\omega = v_1 \wedge \cdots \wedge v_k \in \hat{Gr}(k,V)$, we will 
denote the corresponding subspace $\Span\{v_1,\ldots,v_k\} \subseteq V$ as $L_{\omega}$. If $\dim V = n$, we will sometimes write $\hat{Gr}(k,n)$ instead of $\hat{Gr}(k,V)$.
Choosing a basis $(e_1,\ldots,e_n)$ of $V$ induces coordinates $\{x_I \mid I \subset \{1,\ldots,n\}, |I| = k\}$ on $\bigwedge^k V$, which are known as the \emph{Pl\"ucker coordinates}. The Grassmann cone is a subvariety of $\bigwedge^k V$, and its defining equations are quadrics referred to as the \emph{Pl\"ucker relations} \cite[(1.24)]{Shafarevich}. In the case where $k=2$ and $V=\KK^4$, there exists only one Pl\"ucker relation, called the Klein quadric
$$
     P_{2,4}=x_{1,2}x_{3,4}-x_{1,3}x_{2,4}+x_{1,4}x_{2,3}.
$$
The \emph{Grassmannian} is the projectivization of the Grassmann cone: 
$$
    Gr(k,V) := \PP(\hat{Gr}(k,V)) = \left(\hat{Gr}(k,V) \setminus \{0\} \right) / \KK^{\times} \subseteq \PP\bigg(\bigwedge \nolimits ^k V\bigg).
$$
It is a projective variety whose defining equations are the
Pl\"ucker relations. %

\begin{definition}
A linear map $\varphi: \bigwedge^k V \to \bigwedge^q W$ is \emph{Grassmann cone preserving} (GCP) if 
$$
\varphi(\hat{Gr}(k,V)) \subseteq \hat{Gr}(q,W).
$$
\end{definition}

\begin{example} \label{ex:contraction-map}
    We give two examples of GCP maps.
    \begin{enumerate}
    \item If $f:V \to W$ is a linear map, the induced map $ \bigwedge^kf: \bigwedge^k V \to \bigwedge^k W$ is Grassmann cone preserving. 
    \item %
    For $\beta \in V^*$ the \textit{contraction map}
    \[
	   \varphi_\beta:\bigwedge \nolimits ^k V \to \bigwedge \nolimits ^{k-1}\ker{\beta} \subseteq \bigwedge \nolimits ^{k-1}V
    \]
    defined as
    \[
    	v_1 \wedge \cdots \wedge v_k \mapsto \sum_{i=1}^k (-1)^{i-1} \beta(v_i) v_1 \wedge \cdots \wedge \widehat{v_i} \wedge \cdots \wedge v_k
    \]
    is GCP. Below we give a coordinate-independent description of $\varphi_\beta$.
    \begin{proof}
        Let $v_1 \wedge \cdots \wedge v_k \in \hat{Gr}(k,V) \setminus \{0\}$. Note that he vectors $v_1, \dots, v_k$ are linearly independent. We distinguish two cases. First, assume that ${\rm span}\{v_1,\dots,v_k\} \subseteq \ker{\beta}$. Then, $ \varphi_\beta(v_1 \wedge \cdots \wedge v_k)=0 \in \hat{Gr}(k-1,\ker{\beta})$. Now consider the case that $\ker{\beta} \ \cap \ \operatorname{span} \{v_1,\dots,v_k\}$ has dimension $k-1$. After possibly replacing $v_1,\dots,v_k$ with some $v_1^{\prime},\dots,v_k^{\prime}$ such that $v_1^{\prime} \wedge \cdots \wedge v_k^{\prime}=v_1 \wedge \cdots \wedge v_k$, we can assume that $v_1 \notin \ker{\beta}$, but $v_2,\dots,v_k \in \ker{\beta}$. Then, $\varphi_\beta(v_1 \wedge \cdots \wedge v_k)=\beta(v_1) v_2 \wedge \cdots \wedge v_k$ is contained in $\hat{Gr}(k-1,\ker{\beta})$. This proof also shows that $\varphi_\beta$ takes values in $\bigwedge^{k-1}\ker{\beta}$.
    \end{proof}
    \end{enumerate}

The contraction map $\varphi_\beta$ can also be described coordinate-independently. Recall that there is a natural isomorphism $\bigwedge^k V \cong \operatorname{Alt}^k(V^\ast)$, where $\operatorname{Alt}^k(V^\ast)$ is the space of alternating multilinear maps $V^* \times \dots \times V^* \to \KK$. Under this identification, $\varphi_\beta$ agrees with the map 
\[
	\operatorname{Alt}^k(V^*) \to \operatorname{Alt} ^{k-1}(V^*), \quad \omega \mapsto \omega(\beta, \blank , \ldots, \blank).   
\]

\end{example}

Next, we recall the universality result by Kasman \textit{et al.}

\begin{theorem}[{\cite[Theorem 3.4]{Kasman}}] \label{thm:Kasman}
Let $\omega \in \bigwedge^k V$. Then $\omega \in \hat{Gr}(k,V)$ if and only if every GCP map to $\bigwedge^2 \CC^4$ maps $\omega$ to $\hat{Gr}(2,4)$.
\end{theorem}

In fact, Kasman \textit{et al.} show that the GCP maps can be chosen from an explicit finite collection. 
We can rephrase \Cref{thm:Kasman} in terms of the Klein quadric, as follows. 
\begin{corollary} \label{cor:PullbackOfPlucker} %
Any Grassmannian is set-theoretically defined by pullbacks of the Klein quadric $P_{2,4}$:
\[
    \hat{Gr}(k,V) = \left\{ \omega \in \bigwedge \nolimits^kV \ \bigg| \ P_{2,4}(\varphi(\omega))=0 \quad \forall \varphi \in GCP\left(\bigwedge \nolimits ^k V, \bigwedge \nolimits ^2{\CC^4}\right)\right\}.
\]
\end{corollary}

\section{Quadratic spaces and the isotropic Grassmannian} \label{Sec:hyperbolic spaces}

Throughout the remainder of the paper, we will work in a field $\KK$ of characteristic not $2$. %
In this section, we will introduce quadratic spaces and isotropic Grassmannians, and establish several essential lemmas that we will later use to prove our main theorem.%

\subsection{Quadratic spaces} In this subsection we introduce quadratic spaces. The material is fairly standard, for a reference see %
\cite[Chapter 3]{ArtinGeometricAlgebra}. For \Cref{lemma:anti-orthogonal basis new new} we did not find a proof in the literature, so we opted to give a proof here.

A \emph{quadratic space} refers to a vector space $V$ equipped with a quadratic form---or equivalently, a symmetric bilinear form $\langle \cdot,\cdot \rangle$. We always assume that %
the bilinear form is nondegenerate. 
A vector $v \in V$ is considered \emph{isotropic} if $\langle v,v \rangle=0$. The set of all isotropic vectors in $V$ is denoted by $V_{\iso}$. 
The \emph{orthogonal complement} of a subspace $L \subseteq V$ is defined as the space $L^{\perp}:=\{v \in V \mid \langle v,u \rangle=0, \forall u \in L\}$. 
We call a subspace $L \subseteq V$ \emph{isotropic} if $L \subseteq L^{\perp}$, i.e.\ if $\langle u,v \rangle = 0$ for all $u,v \in L$. By polarization, using $\operatorname{Char}{\KK}\neq 2$, this is equivalent to $L \subseteq V_{\iso}$. %
If $L$ is isotropic but any proper superset $L'\supsetneq L$ is not isotropic, we refer to $L$ as \emph{maximal isotropic}. %

\begin{definition}
We call a tuple $(e_1,e_{-1},\ldots,e_{k},e_{-k})$ of vectors in $V$ \emph{hyperbolic} if $\langle e_i,e_{-i} \rangle=1$ for $i=1,\ldots,k$, and $\langle e_i,e_j \rangle=0$ if $i \neq -j$.
Note that the $e_i$ are necessarily linearly independent. If $2k = \dim V$, then we call $(e_1,e_{-1},\ldots,e_{k},e_{-k})$ a \emph{hyperbolic basis} of $V$.
\end{definition}

\begin{theorem} \label{thm:isotropicComplement}
    Let $L$ be an isotropic subspace of $V$, and $e_1,\ldots,e_k$ a basis of $L$. Then we can find vectors $e_{-1},\ldots,e_{-k} \in V \setminus L$ such that $(e_1,e_{-1},\ldots,e_{k},e_{-k})$ forms a hyperbolic tuple.
\end{theorem}
\begin{proof}
    This is \cite[Theorem 3.8]{ArtinGeometricAlgebra} in the case where $U=L$ is isotropic.
\end{proof}

\begin{theorem}[{See \cite[Theorem 3.10]{ArtinGeometricAlgebra}}]
    All maximal isotropic subspaces of $V$ have the same dimension, which is referred to as the \emph{Witt index} of $V$. 
\end{theorem}
Note that by \Cref{thm:isotropicComplement}, the Witt index can be at most $\left \lfloor{\frac{\dim V}{2}}\right \rfloor$. %
Moreover, this upper bound is attained when $\KK$ is algebraically closed, regardless of the chosen nondegenerate quadratic form.

\begin{convention} \label{conv:maxWitt}
    From this point onward, we make the assumption that $V$ has maximal Witt index $\left \lfloor{\frac{\dim V}{2}}\right \rfloor$. We will denote this Witt index by $p$. 
\end{convention}

\begin{remark} %
    If $\dim V=2p$ is even, then by \Cref{thm:isotropicComplement}, $V$ has a hyperbolic basis. Note that then a subspace $L$ is maximal isotropic if and only if $L=L^{\perp}$. %

    If $\dim V =2p+1$ is odd, then $V$ has a basis
    \begin{equation} \label{eq:hypBasisOdd}
        B=(e_1,e_{-1},\ldots,e_p,e_{-p},e_0) 
    \end{equation}
    such that $(e_1,e_{-1},\ldots,e_p,e_{-p})$ is hyperbolic and $\langle e_0, e_i \rangle = 0$ for all $i \neq 0$. We will call $B$ hyperbolic as well. Note that $\langle e_0,e_0 \rangle \neq 0$ by nondegeneracy. If $\KK$ is algebraically closed we can rescale $e_0$ such that $\langle e_0,e_0 \rangle =1$; in general we will write $c_0 \coloneqq \frac{1}{2}\langle e_0,e_0 \rangle$. Note that we can also find a basis of $V$ consisting of isotropic vectors, for instance by replacing $e_0$ by $e_0 + e_1 - c_0e_{-1}$ in \eqref{eq:hypBasisOdd}.
\end{remark}

The following lemma will be used several times in the proof of our main theorem (to be precise: in \Cref{claim:3steps}, \Cref{Claim intersection nonempty} and \Cref{claim:3cases}).

\begin{lemma}\label{lemma:anti-orthogonal basis new new}
    Let $W_1,W_2 \subseteq V$ be maximal isotropic subspaces. Then for any choice of decomposition
    \[
	   W_1=(W_1 \cap W_2) \oplus U_1 \quad \text{and} \quad W_2=(W_1 \cap W_2) \oplus U_2
    \]
    the isomorphism $V \to V^*, v \mapsto \langle v, \blank \rangle$ restricts to an isomorphism $U_1 \to U_2^*$.
    In particular, there exists a hyperbolic basis $e_1,e_{-1},\dots,e_p,e_{-p},(e_0)$ of $V$, such that  
    \[
	   W_1={\operatorname{span}}\{e_1,\dots,e_p\} \quad \text{and} \quad W_2={\operatorname{span}}\{e_1,\dots,e_q,e_{-(q+1)},\dots,e_{-p}\},
    \]
    where $q={\dim}(W_1 \cap W_2)$.
\end{lemma}

\begin{proof}
    Note that $U_1$ and $U_2$ have the same dimension because the maximal isotropic subspaces $W_1$ and $W_2$ have the same dimension. Thus it suffices to show that the map $U_1 \to U_2^\ast$ is injective. Arguing by contradiction, assume there is some $u_1 \in U_1 \setminus \{0\}$ such that $\langle u_1,u_2 \rangle =0$ for all $u_2 \in U_2$. Then it also holds that $\langle u_1,w_2 \rangle =0$ for all $w_2 \in (W_1 \cap W_2) \oplus U_2=W_2$ because $(W_1 \cap W_2) \oplus \operatorname{span}\{u_1\}\subseteq W_1$ 
    is isotropic. Since $W_2$ and $u_1 \in U_1 \subseteq W_1$ are isotropic, this implies that also $W_2 \oplus {\operatorname{span}}\{u_1\}$ is isotropic. But $W_2$ is strictly contained in $W_2 \oplus {\operatorname{span}}\{u_1\}$ because $u_1 \in U_1 \setminus \{0\}$, contradicting the fact that $W_2$ is maximal isotropic. 
    
    To see how the first statement implies the second one, choose a basis $\{e_1,\dots,e_p\}$ of $W_1$ such that $\{e_1,\dots,e_q\}$ forms a basis for $W_1\cap W_2$. Let $U_1=\Span \{e_{q+1},\ldots,e_p\}$ and choose some $U_2$ such that $W_2=(W_1 \cap W_2) \oplus U_2$. It follows from the first part that there are unique $e_{-(q+1)},\dots,e_{-p} \in U_2$ such that $\langle e_i,e_{-j} \rangle =\delta_{ij}$ for $i,j=q+1,\dots,p$. Since $W_1$ and $W_2$ are isotropic, it holds that %
    $W_1 \cap W_2 \subseteq U_1^{\perp} \cap U_2^{\perp}$,
    and $W_1\cap W_2$ is a maximal isotropic subspace of %
    $U_1^{\perp} \cap U_2^{\perp}$
    by reasons of dimension. So there exist $e_{-1},\dots,e_{-q},(e_0) \in 
    U_1^{\perp} \cap U_2^{\perp}$ %
    such that $(e_1,e_{-1},\dots,e_q,e_{-q},(e_0))$ forms a hyperbolic basis of %
    $U_1^{\perp} \cap U_2^{\perp}$.
    This completes the proof.
\end{proof}

\subsection{The isotropic Grassmann cone}
We now introduce the isotropic Grassmann cone. We continue to work in a quadratic space $V$ satisfying \Cref{conv:maxWitt}.

\begin{definition}
For $k \leq p$, the \emph{isotropic Grassmann cone} is defined as 
$$
	\hat{Gr}_{\iso}(k,V) := \{v_1 \wedge \cdots \wedge v_k \mid \langle v_i,v_j\rangle=0\} \subset \hat{Gr}(k,V) \subset \bigwedge \nolimits ^kV.
$$
If $k=p$, then we call it the \emph{maximal isotropic Grassmann cone}.
Note that $\omega \in \hat{Gr}(k,V)$ lies in $\hat{Gr}_{\iso}(k,V)$ if and only if $L_{\omega} \subset V$ is isotropic. %
\end{definition}

\begin{definition}
A linear map $\Phi: \bigwedge^k V \to \bigwedge^{q} W$ is \emph{isotropic Grassmann cone preserving} (IGCP) if $\Phi(\hat{Gr}_{\iso}(k,V)) \subseteq \hat{Gr}_{\iso}(q,W)$. 
\end{definition}

We will only need one explicit family of IGCP maps; they are the analogue of the GCP maps from \Cref{ex:contraction-map}.
Let $v \in V_{\iso}$ be a non-zero isotropic vector. Define $V_v:=v^{\perp}/ \langle v \rangle$ (note that $\langle v \rangle \subseteq v^\perp$ because $v$ is isotropic). It is easy to see that 
    \[
    	\langle \bar{v}_1,\bar{v}_2 \rangle_{V_v}:=\langle v_1,v_2 \rangle_V,
    \]
    where $\bar{v}_i \in V_v$ denotes the equivalence class of $v_i \in v^\perp$ in $V_v$,
    is a well-defined nondegenerate bilinear form on $V_v$ (i.e., the formula is independent of the choice of representatives $v_1,v_2 \in v^\perp$). Moreover, $(V_v,\langle \cdot, \cdot \rangle_{V_v})$ again has maximal Witt index. We denote by $\pi_v$ the projection $v^{\perp} \twoheadrightarrow V_v$. 

\begin{definition}\label{def:Phi_v}  
    For $v \in V_{\iso} \setminus \{0\}$ we define the linear map
    \[
	   \Phi_v: \bigwedge \nolimits^k V \to \bigwedge \nolimits ^{k-1}V_v
    \]
    as the following composition
    \[
    	\bigwedge\nolimits^k V \xrightarrow{\varphi_v} \bigwedge\nolimits^{k-1}v^\perp \xrightarrow{\bigwedge\nolimits^{k-1}\pi_v}\bigwedge\nolimits^{k-1}V_v,
    \]
    where $\varphi_v$ is the contraction map introduced in \Cref{ex:contraction-map}. 
    Explicitly, this map is given by \begin{equation}\label{phi_coordinates}
        \Phi_v(v_1 \wedge \cdots \wedge v_k) = \sum_{j=1}^{k}(-1)^{j-1} \langle v, v_j \rangle \bar{v}_1 \wedge \cdots \wedge \hat{v}_j \wedge \cdots \wedge \bar{v}_k.
    \end{equation}
    Since $\Phi_v$ is a composition of two GCP maps, it is itself GCP. By the same proof as \Cref{ex:contraction-map}, one readily sees that $\Phi_v$ is in fact IGCP. More explicitly, the following holds.
\end{definition}

\begin{lemma} \label{rule}
 For $\omega \in \hat{Gr}_{\iso}(k,V)$,
\begin{enumerate}
\item if $v \in L_{\omega}^{\perp}$, then $\Phi_{v}(\omega) = 0$,
 		\item if $v \notin L_{\omega}^{\perp}$, then $\Phi_{v}(\omega) \neq  0$, and \[L_{\Phi_{v}(\omega)}=(L_\omega \cap v^{\perp})/{\langle v \rangle}.\]
\end{enumerate}
\end{lemma}

\begin{proof}
    If $v \in L_{\omega}^{\perp}$, then $\langle v, w \rangle = 0$ for all $w \in L_{\omega}$. Therefore, $\varphi_v(\omega)=0$ and hence also $\Phi_v(\omega)=0$. This proves the first statement. For the second statement, suppose $v \notin L_{\omega}^{\perp}$ and choose a basis where $\omega = v_1 \wedge \dots \wedge v_k$, $L_\omega = \operatorname{span} \{v_1, \dots, v_k \}$ and $L_\omega \cap v^{\perp} = \Span\{v_1, \dots, v_{k-1}\}$. 
    By evaluating $\Phi_v(\omega)$ using (\ref{phi_coordinates}), we obtain the result.
 \end{proof}

We now give a more coordinate-independent description of $\Phi_v$. The bilinear form on $V$ induces an isomorphism $V \xrightarrow{\cong} V^*, v \mapsto \langle v, \blank \rangle$. 
Together with the natural isomorphism $\bigwedge^kV^* \cong \operatorname{Alt}^kV$, 
this yields an isomorphism
 $\flat: \bigwedge^k V \xrightarrow{\cong} \operatorname{Alt}^kV$.
 Then $\Phi_v$ is the composition 
 \[
     \bigwedge \nolimits ^k V \xrightarrow{\flat} \operatorname{Alt}^kV \xrightarrow{\Phi_v^{\flat}} \operatorname{Alt}^{k-1}{V_v} \xrightarrow{\flat^{-1}} \bigwedge \nolimits ^{k-1} V_v,
 \]
 where the middle map $\Phi_v^{\flat}$ is given by the formula 
\begin{equation} \label{eq:PhiCoorFree}
\Phi_v^{\flat}(\omega^{\flat})(\bar{v}_1,\ldots,\bar{v}_{k-1}) = \omega^{\flat}(v,v_1,\ldots,v_{k-1}). 
\end{equation}
Note that since $\omega^{\flat}$ is alternating, this does \textit{not} depend on a choice of representatives $v_i \in V$ for $\bar{v}_i \in V_v$.

\subsection{Two lemmas about IGCP maps} We finish this section by proving two lemmas that will play a central role throughout the proof of our main theorem. The first lemma states that no nonzero $\omega$ are annihilated by all IGCP maps: 
\begin{lemma}  \label{lem:nonzero}
    Let $\omega \in \bigwedge^kV$ with $0 < k < \dim V$. If $\Phi_v(\omega)=0$ for all $v \in V_{\iso} \setminus \{0\}$, then $\omega=0$.
\end{lemma}

\begin{proof}
    If $\Phi_v(\omega)=0$ for all $v \in V_{\iso} \setminus \{0\}$, then $\omega^\flat(v,v_2,\dots,v_k)=0$ for all $v \in V_{\iso} \setminus \{0\}$ and $v_2,\dots,v_k \in v^\perp$ due to  \eqref{eq:PhiCoorFree}. But then by \Cref{lem: generated subspace 1} below, %
    $\omega^\flat(w_1,\dots,w_k)=0$ for all $w_1,\dots,w_k \in V$. So $\omega^\flat=0$, and hence $\omega=0$. This completes the proof.
\end{proof}

\begin{prop}
    \label{lem: generated subspace 1}
    If  \(0<k < \dim V \), then the set 
    \[
    	\left\{v \wedge v_2 \wedge \cdots \wedge v_k \in  \bigwedge \nolimits ^k V \, \Big| \, v \in  V_{\iso} \setminus \{0\} \text{ and } v_2,\dots,v_k \in v^\perp\right\}
    \]
    spans \(\bigwedge^k V\).
\end{prop}

\begin{proof}
    Let \(\mathcal{S}\) be the span of the given set in \(\bigwedge^kV\). We choose a hyperbolic basis \(e_1,e_{-1},\dots,e_p,e_{-p},(e_0)\) for \(V\).  It suffices to show that each pure wedge \(e_{i_1} \wedge \dots \wedge e_{i_k}\) is in \(\mathcal{S}\). %
    If there exists \(j \neq 0\) such that \(\# (\{j,-j\}\cap \{i_1,\dots,i_k\})=1\), then clearly \(e_{i_1} \wedge \cdots \wedge e_{i_k} \in \mathcal{S}\). 
    So we only need to show that \(e_{j_1} \wedge e_{-j_1} \wedge \dots \wedge e_{j_m} \wedge e_{-j_m}\) $\in$ \(\mathcal{S}\) when \(k=2m\), or \(e_{j_1} \wedge e_{-j_1} \wedge \cdots \wedge e_{j_m} \wedge e_{-j_m} \wedge e_0\) $\in$ \(\mathcal{S}\) when \(k=2m+1\), where \(j_1,...,j_m \in \{1,...,p\}\). 
   
   If $m<p$, we choose $j_0 \in \{1, \dots, p\} \setminus\{j_1, \dots, j_m\}$.
    We define \(\eta= e_{j_2} \wedge e_{-j_2} \wedge \cdots \wedge e_{j_m} \wedge e_{-j_m}\) if \(k\) is even, and \(\eta= e_{j_2} \wedge e_{-j_2} \wedge \cdots \wedge e_{j_m} \wedge e_{-j_m} \wedge e_0\) if \(k\) is odd. 
    Based on the definition of \(j_0\) and \(\mathcal{S}\), we have
    \(
    	(e_{j_0}+e_{j_1}) \wedge (e_{-j_0}-e_{-j_1})\wedge \eta \in \mathcal{S}.
    \)
    Expanding this expression, we obtain: 
    \[
    	 (e_{j_0}+e_{j_1}) \wedge (e_{-j_0}-e_{-j_1})\wedge \eta = (e_{j_0} \wedge e_{-j_0} - e_{j_1} \wedge e_{-j_1}) \wedge \eta + (\text{terms in \(\mathcal{S}\)}).
    \]
    Therefore, we conclude that 
    \begin{equation}\label{gen subspace - eq1}
		(e_{j_0} \wedge e_{-j_0} - e_{j_1} \wedge e_{-j_1}) \wedge \eta \in \mathcal{S}.
    \end{equation}
    Similarly, by considering
    \(
    	(e_{j_0}+e_{-j_1}) \wedge (e_{-j_0}-e_{j_1}) \wedge \eta \in \mathcal{S},
    \)
    we can deduce 
    \begin{equation}\label{gen subspace - eq2}
		(e_{j_0} \wedge e_{-j_0} -e_{-j_1} \wedge e_{j_1}) \wedge \eta \in \mathcal{S}.
    \end{equation}
    Subtracting %
    \eqref{gen subspace - eq1} from \eqref{gen subspace - eq2}
    and using the anti-symmetry of \(\wedge\), we obtain \(2  e_{j_1} \wedge e_{-j_1} \wedge \eta \in \mathcal{S}\). 
    Since $\Char(\KK) \neq 2$
    this shows that \(e_{j_1} \wedge e_{-j_1} \wedge \ldots \wedge e_{j_m} \wedge e_{-j_m} (\wedge e_0) \in \mathcal{S}\). 

    We still need to consider the case $m=p$; i.e.\ to show that $e_1 \wedge e_{-1} \wedge \dots \wedge e_{p} \wedge e_{-p}$ $\in \mathcal{S}$ if $\dim V = 2p+1$. For this we write $\eta = e_2 \wedge e_{-2} \wedge \dots \wedge e_{p} \wedge e_{-p}$ as before, and note that 
    \[
    2c_0 e_1 \wedge e_{-1} \wedge \eta =\big((e_0+e_1-c_0e_{-1}) \wedge (e_1+c_0e_{-1})-e_0\wedge e_1-c_0e_0\wedge e_{-1}\big)\wedge \eta \in \mathcal{S},
    \]
    where $c_0=\frac{1}{2}\langle e_0,e_0\rangle$. %
\end{proof}

The second lemma is a more technical variant of \Cref{lem:nonzero}. %
We will use it to prove Claims \ref{prop:even} and \ref{claim:oddneq} in the proof of the main theorem.

\begin{lemma} \label{lem:nonzeroNew}
 	Assume $p \geq 2$ and $0 < k < \dim V$, and let $\omega \in \bigwedge^k V$ be nonzero. Let $W$  and $W'$ be maximal isotropic subspaces of $V$ with $\dim(W \cap W')=p-1$, and suppose that $\Phi_v(\omega)=0$ for every isotropic $v \in W \cup W'$. 
    \begin{itemize}
 	\item If $k>p$, then $\omega$ is of the form $\alpha \wedge \omega'$, where $\alpha$ lies in the one-dimensional space $\bigwedge^{p+1}(W+W')$.
 	\item If $k\leq p$ then $\omega \in \bigwedge^k(W^{\perp}\cap W'^{\perp})$.
     \end{itemize}
 \end{lemma}

\begin{proof}
We choose a hyperbolic basis of $V$ such that $W= \Span\{e_1,e_2,\ldots,e_p\}$ and $W'=\Span\{e_{-1},e_2,\ldots,e_p\}$. For $\{i_1,\ldots,i_{\ell}\} \subset \{1,-1,\ldots,p,-p,(0)\}$, we will write $V_{\hat{i}_1,\ldots,\hat{i}_{\ell}}$ for $\Span \{e_i \mid i \notin \{i_1,\ldots,i_{\ell}\}\} \subseteq V$.

     We prove by induction on $i=2,\ldots,p+1$ that 
    \begin{equation}\label{eq:nonzeroNew}
 	\omega = e_1 \wedge e_{-1} \wedge e_2 \wedge \cdots \wedge e_{i-1} \wedge \omega'_{i} + \omega''_i,
    \end{equation}
    with 
    \[
        \omega'_i \in \bigwedge \nolimits ^{k-i}{V_{\hat{1},\hat{2},\hat{3},\ldots,\widehat{i-1},-\hat{1}}} \quad \text{and} \quad \omega''_i \in \bigwedge \nolimits ^{k}{V_{\hat{1},-\hat{1}, -\hat{2}, \ldots, -\widehat{i+1}}},
    \]
    and we put the first summand equal to zero if $i > k$. 

     First, let us show that (\ref{eq:nonzeroNew}) holds for $i=2$. 
     Indeed we can write 
    $$
    	\omega=e_1\wedge e_{-1} \wedge \omega_2' + e_1 \wedge \alpha + e_{-1} \wedge \beta + \omega_2''
     $$
    with 
    \begin{align*}
    	\omega'_{2} \in \bigwedge \nolimits ^{k-2}{V_{\hat{1},-\hat{1}}}, && \alpha,\beta \in\bigwedge \nolimits ^{k-1}{V_{\hat{1},-\hat{1}}}, && \omega''_2 \in \bigwedge \nolimits ^{k}{V_{\hat{1},-\hat{1}}}.
    \end{align*}
    By assumption we have 
    \begin{align*}
 	  0 &= \operatorname{\Phi}_{e_1}(\omega) = \operatorname{\Phi}_{e_1}(e_{-1} \wedge \alpha), \\
 	  0 &= \operatorname{\Phi}_{e_{-1}}(\omega) = \operatorname{\Phi}_{e_{-1}}(e_{1} \wedge \beta),
    \end{align*}
     hence $\alpha=\beta=0$.

     Next we assume (\ref{eq:nonzeroNew}) for some $i$, and want to show it for $i+1$.	
     We can write 
     \begin{align*}
 	  \omega_i' &= e_i \wedge \omega_{i+1}' + e_{-i} \wedge \alpha' + \beta' \\
 	  \omega_i'' &= \omega_{i+1}'' + e_{i} \wedge e_{-i} \wedge \alpha'' + e_{-i} \wedge \beta'',
    \end{align*}
    where
     \begin{align*}
    	\omega'_{i+1} &\in \bigwedge \nolimits ^{k-i-1}{V_{\hat{1},\hat{2},\ldots,\hat{i},-\hat{1}}}, &
 	  \omega''_{i+1} &\in \bigwedge \nolimits ^{k}{V_{\hat{1},-\hat{1},-\hat{2},\ldots,-\hat{i}}}, \\
 	  \alpha' &\in \bigwedge \nolimits ^{k-i-1}{V_{\hat{1},\hat{2},\ldots,\hat{i},-\hat{1},-\hat{i}}},&
 	  \alpha'' &\in \bigwedge \nolimits ^{k-2}{V_{\hat{1},\hat{i},-\hat{1},-\hat{2},\ldots,-\hat{i}}},\\
 	  \beta' &\in \bigwedge \nolimits ^{k-i}{V_{\hat{1},\hat{2},\ldots,\hat{i},-\hat{1},-\hat{i}}}, &
 	  \beta'' &\in \bigwedge \nolimits ^{k-1}{V_{\hat{1},\hat{i},-\hat{1},-\hat{2},\ldots,-\hat{i}}}.
    \end{align*}

    We compute $0=\Phi_{e_{i}}(\omega) =  \bar{e}_1 \wedge  \bar{e}_{-1} \wedge  \bar{e}_2 \wedge \cdots \wedge  \bar{e}_{i-1} \wedge \overline{\alpha'} + \overline{\beta''}$, so we can conclude that $\alpha'=\beta''=0$.

    Next we compute
    \begin{align*}
 	  \Phi_{e_{i}-e_{1}}(\omega) =& \Phi_{e_i-e_1}(e_1 \wedge   e_{-1} \wedge e_2 \wedge \cdots \wedge e_{i-1} \wedge e_i \wedge \omega_{i+1}')\\
 	  &+ \Phi_{e_i-e_1}(e_1 \wedge e_{-1} \wedge e_2 \wedge \cdots \wedge e_{i-1} \wedge \beta') \\
   	    &+\Phi_{e_i-e_1}(\omega_{i+1}'')\\
	   &+\Phi_{e_i-e_1}(e_i \wedge e_{-i} \wedge \alpha'').
    \end{align*}
    The first and third summand are zero by \Cref{rule}. So we get 
     \begin{align*}
 	  \Phi_{e_{i}-e_{1}}(\omega) =&\bar{e}_1 \wedge \bar{e}_2 \wedge \cdots \wedge \bar{e}_{i-1} \wedge \overline{\beta'}
 	  -\bar{e}_i \wedge \overline{\alpha''},
    \end{align*}
    so $e_2 \wedge \cdots \wedge e_{i-1} \wedge \beta' = \alpha''$. If we do the analogous computation for $\Phi_{e_{i}-e_{-1}}(\omega)$ we find that $e_2 \wedge \cdots \wedge e_{i-1} \wedge \beta' = -\alpha''$. So $\beta' = \alpha''=0$, and we get
 $$
 	\omega = e_1 \wedge e_{-1} \wedge e_2 \wedge \cdots \wedge e_{i} \wedge \omega'_{i+1} + \omega''_{i+1},
 $$
    which is exactly (\ref{eq:nonzeroNew}) for $i+1$ instead of $i$.

     Finally, note that the case $i=p+1$ is exactly what we want. Indeed we have
 $$
 	\omega = e_1 \wedge e_{-1} \wedge e_2 \wedge \cdots \wedge e_{p} \wedge \omega' + \omega'',
 $$
    with $\omega' \in \bigwedge^{k-p-1}{V_{\hat{1},\hat{2},\ldots,\hat{p},-\hat{1}}}$ and $\omega'' \in \bigwedge^{k}{V_{\hat{1},\hat{2},\ldots,\hat{p},-\hat{1}}}=\bigwedge^k(W^{\perp}\cap W'^{\perp})$. But if $k \leq p$ the first summand is zero, and if $k > p$ the second summand is zero since $\dim{V_{\hat{1},\hat{2},\ldots,\hat{p},-\hat{1}}}=\dim{V}- p-1\leq p < k$.
    \end{proof}

\section{Universality for maximal isotropic Grassmannians} \label{sec:MainResult}

\subsection{Statement and consequences of the main result} \label{subsec:StaementMainResult}

For this entire section, let $V$ be a quadratic space of maximal Witt index $p = \lfloor\frac{\dim V}{2}\rfloor$ over a field $\KK$ of characteristic not 2. %
\begin{maintheorem} \label{thm:mainTheorem}
    Assume $\dim V > 8$ and let $\omega \in \bigwedge^p V$. If for every isotropic vector $v \in V_{\iso}$, the image of $v$ under the isotropic Grassmann cone preserving map $\Phi_{v}$ lies in $\hat{Gr}_{\iso}(p-1,V_v)$, then $\omega$ itself lies in $\hat{Gr}_{\iso}(p,V)$.
\end{maintheorem}

\begin{corollary} \label{cor:mainCorIsotropic} 
For any $\omega \in \bigwedge^pV$, it holds that $\omega \in \hat{Gr}_{\iso}(p,V)$ if and only if %
\begin{itemize}
    \item every IGCP map to $\bigwedge^3 \KK^7$ maps $\omega$ to $\hat{Gr}_{\iso}(3,7)$, if $\dim V = 2p+1$,
    \item every IGCP map to $\bigwedge^4 \KK^8$ maps $\omega$ to $\hat{Gr}_{\iso}(4,8)$, if $\dim V = 2p$.
\end{itemize}
\end{corollary}

\begin{proof}[Proof of {\Cref{cor:mainCorIsotropic}} assuming {\Cref{thm:mainTheorem}}] One direction is straightforward from the definition of an IGCP map. To prove the other direction, we consider three cases depending on the dimension of V.

If $\operatorname{dim} V  > 8$, we can repeatedly apply the \Cref{thm:mainTheorem} to obtain the desired result. If $\operatorname{dim} V =8$ or $\operatorname{dim} V =7$, we can just apply the assumption to the identity map (which is trivially IGCP). 
For $\operatorname{dim} V <7$, we observe that the map $\varphi: \bigwedge^p V \to \bigwedge^{p+1}{(V \oplus \operatorname{span}\{e_{p+1},e_{-p-1}\})}$, which sends $\omega$ to $\omega \wedge e_{p+1}$, has the property that $\omega$ lies in the isotropic Grassmann cone if and only if $\varphi(\omega)$ lies in the isotropic Grassmann cone. By applying these maps iteratively until we reach $\bigwedge^3\mathbb{K}^7$ or $\bigwedge^4\mathbb{K}^8$, we complete the proof. 
\end{proof}

Similar to \cite[Theorem 4.1]{Kasman}, we obtain a statement about the ranks of quadrics defining the isotropic Grassmann cone, where we use the fact that $\hat{Gr}_{\iso}(p,2p)$ has two irreducible components \cite[Theorem 22.14]{harris}. %

\begin{corollary} \label{rank4}
The isotropic Grassmannian $\hat{Gr}_{\iso}(p,2p+1)$ in its Pl\"ucker embedding can be defined by quadrics of rank at most $4$. Furthermore, both irreducible components of $\hat{Gr}_{\iso}(p,2p)$ can be defined by linear equations and quadrics of rank at most $4$.
\end{corollary}

\begin{proof}
By \Cref{cor:PullbackOfPlucker} it suffices to show the statement is true for $\hat{Gr}_{\iso}(3,7)$ and $\hat{Gr}_{\iso}(4,8)$. This can be done by an explicit calculation, see \Cref{appendix:RankOfQuadrics}. 
\end{proof}

\begin{remark}
The statement in \Cref{rank4} can also be deduced using the Cartan embedding, see \Cref{appendix:RankOfQuadrics}.
\end{remark}

\begin{remark}
A natural question arises: is there a similar result if we replace the symmetric form with a skew-symmetric form, focusing on Lagrangian Grassmannians? The answer, in the case of considering only the Lagrangian Grassmann cone preserving (LGCP) maps $\Phi_v$ for $v \in V$, defined as in \Cref{def:Phi_v}, is no. 

To illustrate this, let us consider an 8-dimensional vector space $V$ with basis $(e_1,\ldots, e_{-4})$ and skew-symmetric form given by $\langle e_i, e_{-i} \rangle = 1$ for $i>0$, $\langle e_{i}, e_{-i} \rangle = -1$ for $i<0$, and all other pairings equal to $0$.

Now, consider the vector $\alpha=e_1\wedge e_{-1}+e_2 \wedge e_{-2}+e_3 \wedge e_{-3}+e_4 \wedge e_{-4}$, and define 
\[
\omega = \alpha \wedge \alpha = 2\sum_{1\leq i<j \leq 4}{e_i\wedge e_{-i} \wedge e_j \wedge e_{-j}} \in \bigwedge \nolimits ^4V.
\]

It can be observed that $\omega$ does not lie in the Grassmann cone since $\omega \wedge \omega$ is a nonzero multiple of $e_1\wedge e_{-1} \wedge e_2\wedge e_{-2} \wedge e_3\wedge e_{-3} \wedge e_4\wedge e_{-4}$. However, upon explicit computation, it can be seen that every LGCP map $\Phi_v$ maps $\omega$ to zero, and thus it lies in the Lagrangian Grassmann cone. 

This example can be generalized to any space of dimension $4m$ by considering $\omega = \alpha^{\wedge m} \in \bigwedge^{2m}V$. Hence, we have a counterexample to the analogue of the \Cref{thm:mainTheorem} (and even to the analogue of \Cref{lem:nonzero}). However, it is not yet a counterexample to the analogue of \Cref{cor:mainCorIsotropic}, as there might be additional LGCP maps that could be considered.

\end{remark}

\subsection{Structure of the proof}
The aim of this subsection is twofold. First, we aim to prove \Cref{thm:cases}, which will serve as the key ingredient in proving the \Cref{thm:mainTheorem}. Secondly, we will give an outline of the proof of the \Cref{thm:mainTheorem} to make it more accessible, as it involves some technical aspects.

We assume $p\geq 2$. 
Note that we can always decompose $V$ as
\begin{equation}
     V=V' \oplus \Span\{e_p,e_{-p}\},
\end{equation}
where $(e_{p},e_{-p})$ is a hyperbolic tuple, 
and $\Span \{e_{p},e_{-p}\} \subset V'^{\perp}$ where $V'$ again has maximal Witt index. For the remaining part of this section we are fixing this decomposition. Any $\omega \in \bigwedge^pV$ can be uniquely written as 
\begin{equation} \label{eq:omega1234}
	\omega =  \omega_1 \wedge e_p \wedge e_{-p} +  \omega_2 \wedge e_{p} + \omega_3 \wedge e_{-p} + \omega_4
\end{equation}
where $\omega_1 \in \bigwedge^{p-2}V'$, $\omega_2,\omega_3 \in \bigwedge^{p-1}V'$ and $\omega_4 \in \bigwedge^{p}V'$.
The following observation shows that for $v \in V'$, a decomposition of $\omega$ maps to a decomposition of $\Phi_v(\omega)$.

\begin{obs}\label{lem:Wprime}
Let $\omega$ be as in (\ref{eq:omega1234}). Then for any $v \in V'$ we have 
\begin{equation} \label{eq:omega1234Prime}
	\Phi_v(\omega) =: \omega' =  \omega'_1 \wedge \bar{e}_p \wedge \bar{e}_{-p} +  \omega'_2 \wedge \bar{e}_{p} + \omega'_3 \wedge \bar{e}_{-p} + \omega'_4,
\end{equation}
where $\omega'_i = \Phi_v(\omega_i)$.
\end{obs}

Next, we give conditions for $\omega$ to be in the isotropic Grassmann cone.

\begin{prop} \label{thm:cases}
Suppose we have written $\omega \in \bigwedge^p V$ in the form (\ref{eq:omega1234}). Assume $\omega \in \hat{Gr}_{\iso}(p,V)$, then one of the following holds:
\begin{enumerate}
	\item \label{(1)} $\omega_1=\omega_3=\omega_4=0$ and $\omega_2 \in \hat{Gr}_{\iso}(p-1,V')$,
	\item \label{(2)} $\omega_1=\omega_2=\omega_4=0$ and $\omega_3 \in \hat{Gr}_{\iso}(p-1,V')$,
	\item \label{(3)} $\omega_1=0$, and $\omega_2$, $\omega_3$, $\omega_4$ are nonzero. Then
		\begin{itemize}
			\item $\omega_2, \omega_3 \in \hat{Gr}_{\iso}(p-1,V')$, $\omega_4 \in \hat{Gr}(p,V')$,
			\item $L_{\omega_2} = L_{\omega_3} \subseteq L_{\omega_4}$.
		\end{itemize}This case only occurs if $\dim V$ is odd.
	\item \label{(4)} $\omega_1$, $\omega_2$, $\omega_3$, $\omega_4$ are all nonzero. Then %
		\begin{itemize}
			\item $\omega_1 \in \hat{Gr}_{\iso}(p-2,V')$, $\omega_2, \omega_3 \in \hat{Gr}_{\iso}(p-1,V')$, $\omega_4 \in \hat{Gr}(p,V')$,
			\item $L_{\omega_2} \cap L_{\omega_3} = L_{\omega_1}$ and $L_{\omega_2} + L_{\omega_3} = L_{\omega_4}$.
		\end{itemize}		
	\end{enumerate}
\end{prop}

\begin{proof}
We define $L' := L_{\omega} \cap V'$. Note that $p-2 \leq \dim L' \leq p-1$, where the second inequality holds since $L'$ is an isotropic subspace of $V'$.
We proceed by considering cases based on $\dim L'$. More precisely, we will show that \eqref{(1)}, \eqref{(2)} or \eqref{(3)} hold if $\dim L'=p-1$, and that \eqref{(4)} holds if $\dim L'=p-2$. %

\underline{Case 1.} %
If $\dim L' = p-1$, then $L'$ is a maximal isotropic subspace of $V'$. Since $L'$ has codimension one in $L_{\omega}$, there exists a vector $v \in L_{\omega}$ such that $L_{\omega} = L' + \Span\{v\}$. Since $L_{\omega}$ is isotropic, we have $L_{\omega} \subseteq L_{\omega}^\perp \subseteq L'^\perp$. 
Therefore, $v \in L'^\perp$. We can write $v = w + v'$, where $w \in \Span\{e_p, e_{-p}\}$ and $v' \in V'$. Moreover, note that $w \in L'^\perp$, and therefore, $v' = v-w \in L'^\perp$. If $\dim V$ is even, then $L'=L'^\perp \cap V'$, hence $v' \in L'$. Consequently, we have 
$$
	L_{\omega} = L' + \Span\{w + v'\} = L' + \Span\{w\}.
$$
Since $\omega$ is isotropic, the vector $w$ is also isotropic. Thus, we can conclude that either 
$$
	w \in \Span\{e_p\} \quad \text{or} \quad w \in \Span\{e_{-p}\}.
$$
So we conclude
$$
	L_{\omega} = L' + \Span\{e_p\} \quad \text{or} \quad L_\omega = L' + \Span\{e_{-p}\},
$$
and therefore 
$$
	\omega = \omega_2 \wedge e_p \quad \text{or} \quad \omega = \omega_3 \wedge e_{-p},
$$
where $\omega_2, \omega_3 \in \hat{Gr}_{\iso}(p-1,V')$.

If $\dim V$ is odd, there is a possibility that $v' \notin L'$. Nevertheless, we still have 
$$
	L_{\omega} = L' + \Span\{w + v'\}.
$$
Writing $w=\lambda e_p + \mu e_{-p}$, we obtain that $2\lambda\mu + \langle v',v' \rangle = 0$. Since $L'$ is maximal isotropic in $V'$, the vector $v'$ cannot be isotropic. Hence, we have $\lambda \neq 0 \neq \mu$. 
Consequently, we can write 
$$
	\omega = \omega' \wedge (\lambda e_p + \mu e_{-p} + v'),
$$
where $L_{\omega'}=L'$. By doing so, we have expressed $\omega$ in the form (\ref{eq:omega1234}), with $\omega_1=0$, $\omega_2=\lambda \omega'$, $\omega_3=\mu \omega'$, and $\omega_4=\omega'\wedge v'$. One can verify that this proves all the claims in \eqref{(3)}.

\underline{Case 2.} 
If $\dim L' = p-2$, we can write $L_\omega = \Span\{e_p+u, e_{-p}+v\} \oplus L'$ for some $u,v \in V'$. We choose $(v_1, \dots, v_{p-2})$ as a basis of $L'$ and express $\omega$ as
\begin{align*}
    \omega &=  v_1 \wedge \cdots \wedge v_{p-2} \wedge (e_p + u) \wedge (e_{-p} + v)\\
        &=: \omega_1 \wedge (e_p + u) \wedge (e_{-p} + v)\\
  		& = \omega_1 \wedge e_p \wedge e_{-p}  - \omega_1 \wedge v \wedge e_p + \omega_1 \wedge u \wedge e_{-p}  + \omega_1 \wedge u \wedge v \\
  		& =: \omega_1 \wedge e_p \wedge e_{-p} + \omega_2 \wedge e_p + \omega_3 \wedge e_{-p} + \omega_4.
\end{align*}
To show that all $\omega_i$ are nonzero, we need to show that $\{v_1, v_2, \dots, v_{p-2}, u, v\}$ are linearly independent. We already know that $\{v_1, \dots, v_{p-2}\}$ are linearly independent since they form a basis of $L'$. Furthermore, $v$ is also linearly independent from $\{v_1, \dots, v_{p-2}\}$; otherwise $e_{-p} \in L_{\omega}$, but this would imply $\langle e_p + u, e_{-p}  \rangle = 0$. Hence, we need to show that $u$ is linearly independent from $\{v_1, \dots, v_{p-2}, v\}$. Assuming $u = \lambda v + v'$, where $v' \in L'$, we obtain
$$
 	\omega = v_1 \wedge \cdots \wedge v_{p-2} \wedge (e_p + \lambda v) \wedge (e_{-p} + v),
$$
where the vectors $\{v_1, \dots, v_{p-2}, e_p + \lambda v, e_{-p} + v\}$ are all isotropic. In particular, the pairing $\langle v, v \rangle = 0$. However, this implies $\langle e_p + \lambda v, e_{-p} + v \rangle = 1$, contradicting the isotropy of $L_\omega$. Hence, the vectors $\{v_1, v_2, \dots, v_{p-2}, u, v\}$ are linearly independent, implying that all $\omega_i$ are nonzero. Note that by definition all $\omega_i$ belong to the corresponding Grassmann cone. Furthermore, since $u$ and $v$ are isotropic and $\langle v_i, v_j \rangle = 0$, $\langle u, v_i \rangle=0$, $\langle v, v_i \rangle=0$ for all $i,j$, we can conclude that $\omega_1, \omega_2$ and $\omega_3$ are isotropic. This proves the first statement in \eqref{(4)}. The second statement follows from the definition and linear independence of $\{v_1, v_2, \dots, v_{p-2}, u, v\}$.
\end{proof}

For the proof of the \Cref{thm:mainTheorem}, we will fix $\omega \in \bigwedge^pV$ satisfying the assumption. We decompose $\omega$ as in \eqref{eq:omega1234}.
Then $\omega$ has one of the following zero patterns:
\[
    \begin{tabular}{ccccc|ccccc}
    & $\omega_1$ & $\omega_2$ & $\omega_3$ & $\omega_4$ & $\omega_1$ & $\omega_2$ & $\omega_3$ & $\omega_4$ & \\
    \hline
    (0) & $0$ & $0$ & $0$ & $0$ & $*$ & $0$ & $0$ & $0$ & (8)\\
    (1) & $0$ & $0$ & $0$ & $*$ & $*$ & $0$ & $0$ & $*$ & (9)\\
    (2) & \cellcolor{yellow}$0$ & \cellcolor{yellow}$0$ & \cellcolor{yellow}$*$ & \cellcolor{yellow}$0$ & $*$ & $0$ & $*$ & $0$ & (10) \\
    (3) & $0$ & $0$ & $*$ & $*$ & $*$ & $0$ & $*$ & $*$ & (11)\\
    (4) & \cellcolor{yellow}$0$ & \cellcolor{yellow}$*$ & \cellcolor{yellow}$0$ & \cellcolor{yellow}$0$ & $*$ & $*$ & $0$ & $0$ & (12)\\
    (5) & $0$ & $*$ & $0$ & $*$ & $*$ & $*$ & $0$ & $*$ & (13)\\
    (6) & $0$ & $*$ & $*$ & $0$ & $*$ & $*$ & $*$ & $0$ & (14)\\
    (7) & \cellcolor{pink}$0$ & \cellcolor{pink}$*$ & \cellcolor{pink}$*$ & \cellcolor{pink}$*$ & \cellcolor{yellow}$*$ & \cellcolor{yellow}$*$ & \cellcolor{yellow}$*$ & \cellcolor{yellow}$*$ & (15)\\
    \end{tabular}
\]
The proof splits into the claims \ref{claim:3steps}--\ref{claim:final_case} which are based on the different zero patterns. First, we will show that zero patterns (0), (1), (3), (5), (6), (8)-(14) and (7) (if $V$ is even-dimensional) are not possible:
    \begin{itemize}
         \item \Cref{claim:3steps} shows that %
         the only possible zero patterns are (2), (4), (7) and (15), with (7) only occurring when $V$ has odd dimension.
    \end{itemize}
Note that the highlighted zero patterns align with the cases in \Cref{thm:cases}. We proceed by proving that the \Cref{thm:mainTheorem} is true if $\omega \in \bigwedge^p V$ has one of the highlighted zero patterns %
as follows: 
   \begin{itemize}
       \item \Cref{claim:zero pattern (2) or (4)} proves that the \Cref{thm:mainTheorem} is true if $\omega$ has zero pattern (2) or (4).
        \item \Cref{Claim intersection nonempty} and \Cref{claim:3cases} show that if $\omega$ has zero pattern (7) or (15), there are three possibilities for the dimension of the intersection $L_{\omega_2} \cap L_{\omega_3}$: 
            \begin{itemize}
                \item[(a)] $\dim(L_{\omega_2} \cap L_{\omega_3})=p-2$ (when $\dim V$ is even);
                \item [(b)] $\dim(L_{\omega_2} \cap L_{\omega_3})=p-2$ (when $\dim V$ is odd);
                \item[(c)] $\dim(L_{\omega_2} \cap L_{\omega_3})=p-1$ (when $\dim V$ is odd).
            \end{itemize}
        \item \Cref{prop:even} shows that the \Cref{thm:mainTheorem} holds for case (a).
        \item \Cref{claim:oddneq} shows that the \Cref{thm:mainTheorem} holds for case (b).
        \item  \Cref{claim:final_case} shows that the \Cref{thm:mainTheorem} holds for case (c).
   \end{itemize}

\subsection{Proof of the main theorem} We will now prove the \Cref{thm:mainTheorem} following the strategy we just explained. 
\begin{proof}[Proof of the Main Theorem $\ref{thm:mainTheorem}$]

Let $\omega \neq 0$ satisfy the assumption of \Cref{thm:mainTheorem}. Trivially, $\omega$ cannot have zero pattern (0).
\begin{claim}\label{claim:3steps}
    $\omega$ cannot have zero pattern (1), (3), (5), (6), or (8)-(14). If $\dim V$ is even it also cannot have zero pattern (7). %
\end{claim}

\begin{proof}
    \begin{step}
        If $\omega_1 \neq 0$, then $\omega_2$, $\omega_3$ and $\omega_4$ are also nonzero. In other words, $\omega$ cannot have zero patterns (8)--(14).
    \end{step}
    \begin{proof} 
        If $\omega_1 \neq 0$, according to \Cref{lem:nonzero}, there exists $v \in V'_{iso}$ such that $\operatorname{\Phi}_v(\omega_1) \neq 0$. Therefore, applying case (\ref{(4)}) of \Cref{thm:cases} to 
        \[
    	   \operatorname{\Phi}_v(\omega)= \operatorname{\Phi}_v(\omega_1) \wedge e_p \wedge e_{-p} +  \operatorname{\Phi}	_v(\omega_2) \wedge e_{p} + \operatorname{\Phi}_v(\omega_3) \wedge e_{-p} + \operatorname{\Phi}_v(\omega_4)
        \]
        we can conclude that  $\operatorname{\Phi}_v(\omega_2)$, $\operatorname{\Phi}_v(\omega_3)$ and $\operatorname{\Phi}_v(\omega_4)$ are nonzero. This implies that $\omega_2$, $\omega_3$ and $\omega_4$ are nonzero as well.
    \end{proof}
    \begin{step}
     If $\omega_4 \neq 0$, then either $\omega$ has zero pattern (15), or $\dim V$ is odd and $\omega$ has zero pattern (7). In other words, $\omega$ cannot have zero patterns (1), (3), (5), and also not (7) if $\dim V$ is even.
    \end{step}
    \begin{proof} 
        As before, by \Cref{lem:nonzero} there exists $v \in V'_{\iso}$ such that
       $\operatorname{\Phi}_v(\omega_4) \neq 0$. The result follows by applying \Cref{thm:cases} to $\operatorname{\Phi}_v(\omega)$ as before.
    \end{proof}
    \begin{step}
        $\omega$ cannot have zero pattern (6).
    \end{step}
    \begin{proof} %
        Assume $\omega$ has zero pattern (6). Our goal is to find a vector $v \in V'_{\iso}$ such that $\Phi_{v}(\omega_2) \neq 0 \neq \Phi_{v}(\omega_3)$. Then $\omega' := \Phi_{v}(\omega)$ also has the property that $\omega'_1=0=\omega'_4$ but $\omega'_2 \neq 0 \neq \omega'_3$. So by \Cref{thm:cases} $\omega'$ is not in $\hat{Gr}_{\iso}(p-1,V_v)$, which is a contradiction with the assumption of the \Cref{thm:mainTheorem}.	
        We consider two cases: 
            \begin{case}
                Assume $L_{\omega_2} + L_{\omega_3} \subsetneq V'$. This case holds if $\dim V$ is odd, and also if  $\dim V$ is even except when $L_{\omega_2} \cap L_{\omega_3} = 0$. Since $V$ is spanned by isotropic vectors, we can find an isotropic vector $v$ that does not lie in the linear subspace $L_{\omega_2} + L_{\omega_3}$. Then we have
                the desired property that $	\Phi_{v}(\omega_2) \neq 0 \neq \Phi_{v}(\omega_3)$ by \Cref{rule}.
            \end{case} 
            
            \begin{case}
    	       Assume $L_{\omega_2} + L_{\omega_3} = V'$. In this case, $\dim V$ is even and $L_{\omega_2} \cap L_{\omega_3} = 0$. By \Cref{lemma:anti-orthogonal basis new new} we can choose a hyperbolic basis such that
	           \[
	                \omega_2 = \alpha e_1 \wedge \cdots \wedge e_{p-1} \quad \text{and} \quad \omega_3 = \beta e_{-1} \wedge \cdots \wedge e_{-p+1}.
                \]
                Taking $v:= e_1 + e_{-2}$, we have $\Phi_{v}(\omega_2) \neq 0 \neq \Phi_{v}(\omega_3)$, satisfying the desired property.
                \qedhere
            \end{case} 
    \end{proof} 
    \noindent
    We now have considered all cases, and the proof of \Cref{claim:3steps} is complete.
\end{proof} 

We now know that $\omega$ has one of the highlighted zero patterns. Next, we prove that the \Cref{thm:mainTheorem} holds if $\omega$ has zero pattern (2) or (4).

\begin{claim} \label{claim:zero pattern (2) or (4)}
	The \Cref{thm:mainTheorem} is true if $\omega$ has zero pattern (2) or (4).
\end{claim}
\begin{proof} 
    Let $\omega$ have zero pattern (2). Then $\omega = \omega_3 \wedge e_{-p}$. For $v=e_{p}$, by (\ref{phi_coordinates}), we have $\Phi_{e_{p}}(\omega) = \pm \omega_3$, which by assumption lies in $\hat{Gr}_{\iso}(p-1,V_v) = \hat{Gr}_{\iso}(p-1,V')$  and therefore also $\omega \in \hat{Gr}_{\iso}(p,V)$. If  $\omega$ has zero pattern (4) we proceed analogously, using $v=e_{-p}$.%
\end{proof}

For the rest of the proof, we assume that $\omega$ has zero pattern (15) or (7); where (7) can only occur if $\dim V$ is odd. %

\begin{claim} \label{Claim intersection nonempty}
	The intersection $L_{\omega_2} \cap L_{\omega_3}$ is nonzero.
\end{claim}

\begin{proof}
    Assume by contradiction that $L_{\omega_2} \cap L_{\omega_3}=0$. Then by \Cref{lemma:anti-orthogonal basis new new} we can find a hyperbolic basis of $V'$ such that
    \[
    	\omega_2 = \alpha e_1 \wedge \cdots \wedge e_{p-1} \quad \text{and} \quad \omega_3 = \beta e_{-1} \wedge \cdots \wedge e_{-p+1}.
    \]	
    If $\dim V$ is even we take $v=e_1+e_{-2}$. By \Cref{lem:Wprime} we get
    \[
    	\Phi_v(\omega) =: \omega' = \omega'_1 \wedge \bar{e}_p \wedge \bar{e}_{-p} + \omega'_2 \wedge\bar{e}_{p} + \omega'_3\wedge \bar{e}_{-p} + \omega'_4.
    \]
    Note that in the quotient space $ \bigwedge^{p-1}{(e_1 + e_{-2})}^{\perp}/{\langle e_1 + e_{-2} \rangle}$, $\omega'_2$ and $\omega'_3$ have only the basis vector $\bar{e}_{1}=\bar{e}_{-2}$ in common, thus 
    \[
    	\dim(L_{\omega'_2} \cap L_{\omega'_3}) = 1. 
    \]
    Since $\omega' \in \hat{Gr}_{\iso}(p-1,V_v)$, we can conclude by \Cref{thm:cases} that the intersection $L_{\omega'_2} \cap L_{\omega'_3} = L_{\omega'_1}$, in particular 
    \[
        \dim(L_{\omega'_2} \cap L_{\omega'_3})=p-3.
    \]
    This contradicts our assumption $\dim V > 8$, which for $\dim V$ even implies $p>4$.
	
    If $\dim V$ is odd we take $v=e_0+e_{-1}-c_0e_1$, where $c_0 = \frac{1}{2}\langle e_0, e_0\rangle$. Then we find
    \[
    	\omega_2'= \alpha \bar{e}_2 \wedge \cdots \wedge \bar{e}_{p-1} \quad \text{and} \quad 
    	\omega_3' = -c_0\beta \bar{e}_{-2} \wedge \cdots \wedge \bar{e}_{-p+1}.
    \]
   Note that also in the quotient space $\bigwedge^{p-1}{(e_0+e_{-1}-c_0e_1)}^{\perp}/{\langle e_0+e_{-1}-c_0e_1 \rangle}$, we have that 
    \[
	   \dim(L_{\omega'_2} \cap L_{\omega'_3}) = 0. 
    \]
    But by inspecting cases (\ref{(3)}) and (\ref{(4)}) of \Cref{thm:cases}, we see that
    \[
	   \dim(L_{\omega'_2} \cap L_{\omega'_3})\in \{p-2,p-3\},
    \]
    which again is a contradiction since if $\dim V$ is odd, our assumption $\dim V > 8$ implies $p>3$.
\end{proof}

We will use this result and distinguish if the dimension of $V$ is odd or even.

\begin{claim} \label{claim:3cases}
    One of the following holds.
        \begin{enumerate}[label=(\alph*)]
        	\item\label{it:even} $\dim V$ is even, and $\dim(L_{\omega_2} \cap L_{\omega_3})=p-2$,
	       \item\label{it:oddneq} $\dim V$ is odd, and $\dim(L_{\omega_2} \cap L_{\omega_3})=p-2$,
	       \item\label{it:oddeq} $\dim V$ is odd, and $\dim(L_{\omega_2} \cap L_{\omega_3})=p-1$ and thus $L_{\omega_2} = L_{\omega_3}$.
        \end{enumerate} 
\end{claim}

\begin{proof}
    We write $q \coloneqq \dim (L_{\omega_2} \cap L_{\omega_3}) >0$. By \Cref{lemma:anti-orthogonal basis new new} we can find a hyperbolic basis of $V'$ such that
    \[
    	\omega_2 = e_1 \wedge \cdots \wedge e_q \wedge \tilde{\omega}_2 \quad \text{and} \quad \omega_3 = e_1 \wedge \cdots \wedge e_q \wedge \tilde{\omega}_3,
    \]
    where $\tilde{\omega}_2,\tilde{\omega}_3 \in \hat{Gr}_{\iso}(p-q-1, \tilde{V})$ and $\tilde{V} = \Span \{e_{q+1}, e_{-q-1}, \ldots, e_{p-1}, e_{-p+1}\}$. 
    Now if we choose $v$ as any negative indexed basis vector $e_{-i}$ and write $w'=\Phi_v(\omega)$, by \Cref{lem:Wprime} we get
    \begin{align*}
        \omega_2' &= \pm \bar{e}_1 \wedge \cdots \wedge \bar{e}_{i-1} \wedge \hat{e}_i  \wedge \bar{e}_{i+1} \wedge \cdots \wedge \bar{e}_q \wedge \tilde{\omega}_2, \\
        \omega_3' &= \pm \bar{e}_1 \wedge \cdots \wedge \bar{e}_{i-1} \wedge \hat{e}_i  \wedge \bar{e}_{i+1} \wedge \cdots \wedge \bar{e}_q \wedge \tilde{\omega}_3.
    \end{align*}
    Therefore we have $L_{\omega'_2} \cap L_{\omega'_3} = \Span\{\bar{e}_1, \dots, \bar{e}_{i-1}, \hat{e}_{i}, \bar{e}_{i+1}, \dots, \bar{e}_q\}$ and it holds that 
    \[
        \dim(L_{\omega'_2} \cap L_{\omega'_3}) = q-1.
    \]
   If $\dim V$ is even, then
    \[
        \dim(L_{\omega'_2} \cap L_{\omega'_3}) = p-3,
    \]
    hence we conclude $q=p-2$, as desired. Similarly,  if $\dim V$ is odd, then
    \[
	   \dim(L_{\omega'_2} \cap L_{\omega'_3})\in \{p-2,p-3\}.
    \]
    We conclude $q\in \{p-2,p-1\}$. In other words, either $q = p-2$, or $L_{\omega_2} = L_{\omega_3}$. 
\end{proof} \phantom\qedhere

We will finish the proof by a case analysis of the cases in \Cref{claim:3cases}. For the first two cases we need \Cref{lem:nonzeroNew}.

\begin{claim} \label{prop:even}
	The \Cref{thm:mainTheorem} holds in case \ref{it:even}.
\end{claim}
\begin{proof}
    Observe that for every $v \in L_{\omega_2} \cup L_{\omega_3}$,  
    either $\Phi_v(\omega_2)$ or $\Phi_v(\omega_3)$ is zero. By \Cref{thm:cases} and hence the possible zero patterns, this implies that also $\Phi_v(\omega_1)$ and $\Phi_v(\omega_4)$ are zero. Thus, applying \Cref{lem:nonzeroNew} to them yields that $\omega_1 \in \bigwedge^{p-1}(L_{\omega_2}^\perp \cap L_{\omega_3}^\perp) = \bigwedge^{p-1}(L_{\omega_2} \cap L_{\omega_3})$ and $\omega_4 \in \bigwedge^{p+1}(L_{\omega_2} + L_{\omega_3})$. %
    After choosing a hyperbolic basis for $V'$ with $L_{\omega_2}= \Span\{e_1,e_2,\ldots,e_{p-1}\}$ and $L_{\omega_3}= \Span\{e_{-1},e_2,\ldots,e_{p-1}\}$ and defining $W := \Span\{e_1,e_{-1},e_p,e_{-p}\}$, we can write 
    \begin{equation} \label{eq:gr24}
        \omega=e_2 \wedge \cdots \wedge e_{p-1} \wedge \eta
    \end{equation}
    for some $\eta \in \bigwedge^2 W$. Note that $L_{\omega_2} \cap L_{\omega_3} = \Span\{ e_2, \ldots e_{p-1}\}$ is isotropic and orthogonal to $W$. Now, choose $v=e_{-p+1}$. Then we have $V_v \cong V''_v \oplus W$, where $V'' = \Span\{e_2,e_{-2},\ldots,e_{p-1},e_{-p+1}\}$. Using \eqref{eq:gr24} we can write    
    \[
	   \operatorname{\Phi}_v(\omega) = \pm \bar{e}_2 \wedge \cdots \wedge \bar{e}_{p-2} \wedge \eta.
    \]
    By assumption we have $\Phi_v(\omega) \in \hat{Gr}_{\iso}(p-1,V_v)$. So we find that $\eta \in \hat{Gr}_{\iso}(2,W)$, which in turn implies $\omega \in \hat{Gr}_{\iso}(p,V)$.
\end{proof}

\begin{claim} \label{claim:oddneq}
	The \Cref{thm:mainTheorem} holds in case \ref{it:oddneq}.	
\end{claim}

\begin{proof}
   We choose a hyperbolic basis of $V'$ with $L_{\omega_2}= \Span\{e_1,e_2,\ldots,e_{p-1}\}$ and $L_{\omega_3}= \Span\{e_{-1},e_2,\ldots,e_{p-1}\}$. %
    Applying \Cref{lem:nonzeroNew} to $\omega_1$ and $\omega_4$ we get $\omega_1\in \bigwedge^{p-2}{(L_{\omega_2}^{\perp} \cap L_{\omega_3}^{\perp})}$ and $\omega_4 = \nu e_1 \wedge e_{-1} \wedge e_2 \wedge \cdots\wedge e_{p-1}$ for some $\nu \in \KK^\times$. 
    So we can write
    \begin{align*}
	   \omega=&\bigg(\mu_0 e_2 \wedge \cdots \wedge e_{p-1} + e_0\wedge\sum_{i=2}^{p-1}{\mu_ie_2\wedge \cdots \wedge \hat{e}_i \wedge \cdots \wedge e_{p-1}}\bigg) \wedge e_p \wedge e_{-p} \\
	   &+\alpha e_1 \wedge e_2 \wedge \cdots \wedge e_{p-1} \wedge e_p + \beta e_{-1} \wedge e_2 \wedge \cdots \wedge e_{p-1} \wedge e_{-p} \\
	   &+\nu e_1 \wedge e_{-1} \wedge e_2 \wedge \cdots \wedge e_{p-1}.
    \end{align*}
    Picking $v=e_{-2}$ yields
    \begin{align*}
	   \Phi_v(\omega)=&\bigg(\mu_0 e_3 \wedge \cdots \wedge e_{p-1} - e_0\wedge\sum_{i=3}^{p-1}{\mu_i e_3\wedge \cdots \wedge \hat{e}_i \wedge \cdots \wedge e_{p-1}}\bigg) \wedge e_p \wedge e_{-p} \\
	   &-\alpha e_1 \wedge e_3 \wedge \cdots \wedge e_{p-1} \wedge e_p - \beta e_{-1} \wedge e_3 \wedge \cdots \wedge e_{p-1} \wedge e_{-p} \\
	   &+\nu e_1 \wedge e_{-1} \wedge e_3 \wedge \cdots \wedge e_{p-1}.
    \end{align*}

   By assumption $\omega'=\Phi_v(\omega) \in \hat{Gr}_{\iso}(p-1,V_v)$. Thus, by \Cref{thm:cases} one of the cases (\ref{(1)}) -- (\ref{(4)}) holds. 
   Clearly, case (\ref{(1)}) and (\ref{(2)}), and since $L_{\omega_2'} \neq L_{\omega_3'}$, also case (\ref{(3)}), are not possible. Thus, case (\ref{(4)}) holds, which implies that $\omega'_1 \in \hat{Gr}_{\iso}(p-3,V'_v)$ and $L_{\omega_1'} =L_{\omega_2'} \cap L_{\omega_3'}$. In coordinates, this means that $\mu_i = 0$ for $i=3,\ldots, p-1$, and that $\mu_0 \neq 0$. The same argument\footnote{here we use $p\geq 4$} with $v=e_{-3}$ shows that also $\mu_2=0$. We now have written $\omega$ as in (\ref{eq:gr24}), and can proceed exactly as in \Cref{prop:even}. 
\end{proof}

\begin{claim}\label{claim:final_case}
	The \Cref{thm:mainTheorem} holds in case \ref{it:oddeq}. %
\end{claim}

\begin{proof} The proof is divided into several steps: 
\begin{itemize}
        \item Step 1 shows that $\omega_1 = 0$.
        \item Steps 2-4 show that $\omega_4 = e_1 \wedge \cdots \wedge e_{p-1} \wedge u$ for some $u \in V'$.
        \begin{itemize}
        \item Step 2 shows that we can write: 
        \[
            \omega_4=e_1\wedge\cdots\wedge e_{p-1} \wedge u + \sum_{j_1,\ldots,j_{\ell}}{\mu_{J} e_{j_1}\wedge e_{-j_1} \wedge \cdots \wedge e_{j_\ell}\wedge e_{-j_\ell}(\wedge e_{0})},
        \]
        where we write $p=2\ell$ or $p=2\ell+1$, and the factor $\wedge e_{0}$ only appears in the latter case.
        \item Step 3 shows that all $\mu_J$ are equal and thus:
        \[
            \omega_4=e_1\wedge\cdots\wedge e_{p-1} \wedge u + \mu\sum_{j_1,\ldots,j_{\ell}}{e_{j_1}\wedge e_{-j_1} \wedge \cdots \wedge e_{j_\ell}\wedge e_{-j_\ell}(\wedge e_{0})}.
        \]
        \item Step 4 shows that $\mu=0$.
        \end{itemize}
        \item Step 5 then concludes that $\omega \in \hat{Gr}_{\iso}(p,V)$.
\end{itemize}
\setcounter{step}{0}
\begin{step}
    $\omega_1 = 0$
\end{step}
\begin{proof}
 Note that for every $v \in V'_{\iso}$, if we consider $\omega' := \Phi_v(\omega) \in \hat{Gr}_{\iso}(p-1,V_v)$, we have $L_{\omega'_2}=L_{\omega'_3}$. This implies that either $\omega'=0$, or %
 $\omega'$ is in case (\ref{(3)}) of \Cref{thm:cases}. In both cases we conclude $\omega'_1=0$. So we proved $\Phi_v(\omega_1)=0$ for each $v \in V'_{\iso}$, which by \Cref{lem:nonzero} implies that $\omega_1=0$.
\end{proof}
   Now we choose a hyperbolic basis of $V'$ such that $L_{\omega_2} = L_{\omega_3} = \Span\{e_1,\ldots,e_p\}$. For any $v \in V_{\iso}$, we can apply \Cref{thm:cases} to $\Phi_v(\omega)$ and find: 
\begin{enumerate}
    \item\label{it:omega4+} If $v \in L_{\omega_2}$, we have $\Phi_v(\omega_4)=0$.
    \item\label{it:omega4General} If $v \in V_{\iso} \setminus L_{\omega_2}$, we have $\Phi_v(\omega_4) \in  \hat{Gr}(p-1,V_v)$ with $L_{\omega_2} \subset L_{\Phi_v(\omega_4)}$. 
\end{enumerate}
\begin{step} $\omega_4$ is of the form 
\[
	e_1\wedge\cdots\wedge e_{p-1} \wedge u + \sum_{j_1,\ldots,j_{\ell}}{\mu_{J} e_{j_1}\wedge e_{-j_1} \wedge \cdots \wedge e_{j_\ell}\wedge e_{-j_\ell}(\wedge e_{0})},
\]
where we write $p=2\ell$ or $p=2\ell+1$, and the factor $\wedge e_{0}$ only appears in the latter case.
\end{step}
\begin{proof}
We will write
\[
	   \omega_4=\sum_{i_1,\ldots,i_p}{\lambda_{i_1,\ldots,i_p}e_{i_1}\wedge \cdots \wedge e_{i_p}},
\]
where we always order the indices as follows:  $1,-1,2,-2,\ldots, p-1, -p+1,0$. We will abbreviate $\lambda_{i_1,\ldots,i_p}$ to $\lambda_I$, where $I = \{i_1,\ldots,i_p\} \subset \{1,-1,2,-2,\ldots,p-1,-p+1,0\}$.
If we choose $v=e_i$, then \eqref{it:omega4+} tells us that
        \[
            0=\Phi_v(\omega_4)=\sum_{-i\in I, i \notin I}{\pm\lambda_{I}\bar{e}_{I\setminus\{-i\}}} \in \bigwedge \nolimits ^{p-1}{e_i^{\perp} / {\langle e_i \rangle}},
        \]
        where the occurring vectors $\bar{e}_{I\setminus\{-i\}}$ are linearly independent. So if $-i \in I$ but $i \notin I$ then $\lambda_I=0$.
        On the other hand, if we choose $v=e_{-i}$, %
        then \eqref{it:omega4General} tells us that  
        \[
	       \Phi_v(\omega_4)=\sum_{-i\notin I, i \in I}{\pm\lambda_{I}\bar{e}_{I\setminus\{i\}}} \in \bigwedge \nolimits ^{p-1}{e_{-i}^{\perp} / {\langle e_{-i} \rangle}}
        \]
        is of the form 
        \[
            \bar{e}_1 \wedge \cdots \wedge \hat{e}_i \wedge \cdots \wedge \bar{e}_{p-1} \wedge u
        \]
        for some $u \in V'$. So if $i \in I$ but $-i\notin I$, then $\lambda_I=0$, unless $\{1,2, \ldots, p-1\} \subset I$. Together with the above, this implies the claim.
        \end{proof}
    \begin{step}
    All $\mu_I$ are equal, so we can write
    \[
	   \omega_4=e_1\wedge\cdots\wedge e_{p-1} \wedge u + \mu\sum_{j_1,\ldots,j_{\ell}}{e_{j_1}\wedge e_{-j_1} \wedge \cdots \wedge e_{j_\ell}\wedge e_{-j_\ell}(\wedge e_{0})}.
    \]
    \end{step}
    \begin{proof}
    Take $v=e_i-e_j$ with $i,j$ positive. Then $\Phi_{v}(\omega_1)=0$ and $\Phi_{v}(\omega_2)=\Phi_{v}(\omega_3)=0$, hence by \Cref{thm:cases} we get $\Phi_v(\omega_4)=0$. But
    \begin{align*}
        \varphi_v(\omega_4) =& -\sum_{i \in J}{\mu_J e_{j_1}\wedge e_{-j_1} \wedge \cdots \wedge e_{i} \wedge \hat{e}_{-i} \wedge \cdots \wedge e_{j_\ell}\wedge e_{-j_\ell}(\wedge e_{0})}\\
        & +\sum_{j \in J}{\mu_J e_{j_1}\wedge e_{-j_1} \wedge \cdots \wedge e_{j} \wedge \hat{e}_{-j} \wedge \cdots \wedge e_{j_\ell}\wedge e_{-j_\ell}(\wedge e_{0})}.
    \end{align*}
    After projecting to $ \bigwedge^{p-1}{(e_i -  e_{j})}^{\perp}/{\langle e_i - e_{j} \rangle}$ we get
    \begin{align*}
        0 = \Phi_v(\omega_4) =& -\sum_{i \in J, j \notin J}{\mu_J \bar{e}_{j_1}\wedge \bar{e}_{-j_1} \wedge \cdots \wedge {\color{blue} \bar{e}_{i}} \wedge \hat{e}_{-i} \wedge \cdots \wedge \bar{e}_{j_\ell}\wedge \bar{e}_{-j_\ell}(\wedge \bar{e}_{0})}\\
        & +\sum_{j \in J, i \notin J}{\mu_J \bar{e}_{j_1}\wedge \bar{e}_{-j_1} \wedge \cdots \wedge {\color{blue} \bar{e}_{j}} \wedge \hat{e}_{-j} \wedge \cdots \wedge \bar{e}_{j_\ell}\wedge \bar{e}_{-j_\ell}(\wedge \bar{e}_{0})}\\
	    =&{\color{blue} \bar{e}_{i}} \wedge \sum_{i,j \notin J'}(-\mu_{J'\cup\{i\}}+\mu_{J'\cup\{j\}})\bar{e}_{j'_1}\wedge \bar{e}_{-j'_1} \wedge \cdots \wedge \bar{e}_{j'_{\ell-1}}\wedge \bar{e}_{-j'_{\ell-1}}(\wedge \bar{e}_{0}).
    \end{align*}
    So we find that $\mu_{J'\cup\{i\}}=\mu_{J'\cup\{j\}}$ for every $J' \subset \{1, \ldots, \hat{i}, \ldots,\hat{j}, \ldots, p-1\}$. Letting $i$ and $j$ vary yields the result. 
    \end{proof}
    \begin{step}
        $\mu_4=0$
    \end{step}
    \begin{proof}
    Finally, take $v=e_0+e_{-1}-c_0e_{1}$, where as before $c_0=\frac{1}{2}\langle e_0,e_0 \rangle$.  
    If $p=2\ell$ we can write 
    \[
	   \Phi_v(\omega_4)= \bar{e}_2\wedge\cdots\wedge \bar{e}_{p-1} \wedge \tilde{u} + \mu \sum_{J \ni 1}(\bar{e}_{-1}+c_0\bar{e}_1)\wedge \bar{e}_{j_2} \wedge \bar{e}_{-j_2}\wedge \cdots \wedge \bar{e}_{j_\ell} \wedge \bar{e}_{-j_\ell}  
    \]
    and if $p=2\ell+1$ we have 
    \begin{align*}
	   \Phi_v(\omega_4)=& \bar{e}_2\wedge\cdots\wedge \bar{e}_{p-1} \wedge \tilde{u} + \mu \sum_{J \ni 1}(\bar{e}_{-1}+c_0\bar{e}_1)\wedge \bar{e}_{j_2} \wedge \bar{e}_{-j_2}\wedge \cdots \wedge \bar{e}_{j_\ell} \wedge \bar{e}_{-j_\ell} \wedge \bar{e}_0 \\
	   &+ 2c_0\mu \sum_{J}{\bar{e}_{j_1}\wedge \bar{e}_{-j_1} \wedge \cdots \wedge \bar{e}_{j_\ell}\wedge \bar{e}_{-j_\ell}}.  
    \end{align*}
    In both cases we have $L_{\Phi_v(\omega_4)} \supset \Span\{\bar{e}_2,\ldots,\bar{e}_{p-1}\}$ by \eqref{it:omega4General} from which we conclude $\mu = 0$.
\end{proof}
\begin{step}
    $\omega \in \hat{Gr}_{\iso}(p,V)$
\end{step}
\begin{proof}
Since $\omega_1=0$ and $\omega_4 = e_1 \wedge \cdots \wedge e_{p-1} \wedge u$ for some $u \in V'$, we can write
\[
\omega=e_1 \wedge \dots \wedge e_{p-1} \wedge u'
\]
for some $u' \in \Span\{e_0, e_p, e_{-1}, \ldots, e_{-p+1}, e_{-p}\}$. Choose $v=e_{-1}$, then we have 
    \[
        \Phi_v(\omega) = \bar{e}_2 \wedge \dots \wedge \bar{e}_{p-1} \wedge \overline{u'} \in \hat{Gr}_{\iso}(p-1,V_v)
    \] 
hence $\langle u', u' \rangle = 0$ and $\langle e_j, u' \rangle = 0$ for all $j=2, \ldots, p-1$. Replacing $v=e_{-1}$ with $v=e_{-2}$ yields that also $\langle e_1,u'\rangle=0$, hence $\omega \in \hat{Gr}_{\iso}(p,V)$. 
\end{proof}
This proves \Cref{claim:final_case}.
\end{proof}
\end{proof}

\section{Counterexamples in small dimensions} \label{sec:counterexamples}

In the \Cref{thm:mainTheorem}, we assumed that $\dim V > 8$. In this section, we will show that this assumption is actually necessary. In both $\bigwedge^3\KK^7$ and $\bigwedge^4\KK^8$, we will give a $p$-form $\omega$ that does not lie in the isotropic Grassmannian, but which maps to the isotropic Grassmannian upon applying any IGCP map $\Phi_v$. For case of simplicity, we assume the underlying field $\KK$ is either $\CC$ or $\RR$.

\subsection{Counterexample in dimension 7} \label{sec:counterexampleOdd}
Let $V$ be a $7$-dimensional vector space over $\KK$ with a fixed basis $\{e_0,e_1,e_2,e_3,e_{-1},e_{-2},e_{-3}\}$, and a quadratic form given by the matrix
$$
    J=\begin{pmatrix}
     -\frac{1}{2} & \rvline & 0 & \rvline & 0 \\
     \hline
     0 & \rvline & 0 & \rvline & I_3 \\
     \hline
    0 & \rvline & I_3 & \rvline & 0 \\
    \end{pmatrix}.
$$
Choose
$$
	\omega_7:=e_1\wedge e_2 \wedge e_3 + e_{-1} \wedge e_{-2} \wedge e_{-3} + e_0 \wedge (e_1 \wedge e_{-1} + e_2 \wedge e_{-2} + e_3 \wedge e_{-3}). 
$$
One verifies that $\omega_7 \notin \hat{Gr}(3,V)$, so in particular $\omega_7 \notin \hat{Gr}_{\iso}(3,V)$. In \Cref{claim:oddCounterExample} below we will show that every $\Phi_v$ maps $\omega_7$ to the isotropic Grassmann cone $\hat{Gr}_{\iso}(2,5)$. One could verify this by a direct computation for an arbitrary isotropic vector $v$. However, we will exploit the fact that $\omega_7$ is sufficiently symmetric (\Cref{claim:oddCaseTransitive}), so it suffices to do the computation for one fixed $v \in V_{\iso}$.

Consider the algebraic group 
\begin{align*}
SO(V) = & \{\phi \in SL(V) \mid \langle \phi(x),\phi(y) \rangle = \langle x, y \rangle \quad \forall x,y \in V\} \\ = & \{A \in SL(7,\KK) \mid A^TJA=J\}
\end{align*}
and its subgroup
$$
    G=\stab(\omega_7)=\{\phi \in SO(V) \mid \phi \cdot \omega_7 = \omega_7 \}.
$$
\begin{claim} \label{claim:oddCaseTransitive}
The action of $G$ on $V_{\iso}$ is transitive.
\end{claim}

\begin{proof}
Take any $v_0 \in V_{\iso}$. We want to show that its orbit $G \cdot v_0$ has dimension six. Then $G \cdot v_0$ is a full-dimensional subvariety of the irreducible $6$-dimensional variety $V_{\iso}$, and hence is equal to $V_{\iso}$. For this we use the formula 
$$
    \dim(G \cdot v_0) = \dim{G} - \dim(\stab_{G}(v_0)),
$$
where $\stab_{G}(v_0)=\{\phi \in G \mid \phi\cdot v_0 = v_0\}$ is the stabilizer. We will compute both terms $\dim{G}$ and $\dim(\stab_{G}(v_0))$ by switching to Lie algebras.

The Lie algebra of $SO(V)$ is given by
\begin{align*}
\mathfrak{so}(V) =& \{X \in \mathfrak{sl}(7,\KK) \mid X^{T}J+JX=0\} \\ =& \Bigg\{\begin{pmatrix}
 0 & \rvline & -2y^T & \rvline & -2x^T \\
 \hline
 x & \rvline & a & \rvline & b \\
 \hline
 y & \rvline & c & \rvline & -a^T \\
\end{pmatrix} \Bigg \lvert x,y \in \KK^3, a,b,c \in \KK^{3\times3}, b+b^T=c+c^T=0\Bigg\}.
\end{align*}
We introduce the following notation. 
$$
    \text{For }x=\begin{pmatrix}
    x_1 \\ x_2 \\ x_3
    \end{pmatrix}
    \text{we write }
    l_x := \begin{pmatrix}
    0 & -x_3 & x_2 \\ x_3 & 0 & -x_1 \\ -x_2 & x_1 & 0
    \end{pmatrix}.
$$
We can compute the Lie algebra $\mathfrak{g} \subset \mathfrak{so}(V)$ of $G$ as follows:
\begin{align*}
\mathfrak{g} =& \{X \in \mathfrak{so}(V) \mid X\cdot \omega_7 =0\} \\ =& \Bigg\{\begin{pmatrix}
 0 & \rvline & -2y^T & \rvline & -2x^T \\
 \hline
 x & \rvline & a & \rvline & l_y \\
 \hline
 y & \rvline & l_x & \rvline & -a^T \\
\end{pmatrix} \Bigg \lvert x,y \in \KK^3, a \in \mathfrak{sl}(3,\KK)\Bigg\}.
\end{align*}
Observe that $\dim \mathfrak{g} =3+3+8=14$ and since $\dim G=\dim \mathfrak{g}$, $G$ has also dimension $14$. For the stabilizer, if we take $v_0=e_{-3} \in V_{\iso}$, we see that
\begin{align*}
\stab_{\mathfrak{g}}(v_0) =& \{X \in \mathfrak{so}(V) \mid X\cdot e_{-3} =0\}
\end{align*}
is the set of matrices in $\mathfrak{g}$ whose final column is zero, which has dimension $8$. So $\dim(G \cdot v_0) = 14 -8 = 6 = \dim V_{\iso}$, as desired.
\end{proof}

\begin{claim} \label{claim:oddCounterExample}
For every $v \in V_{\iso}$, it holds that $\Phi_v(\omega_7) \in \hat{Gr}_{\iso}(2,V_v)$.
\end{claim}

\begin{proof}
By the previous claim, it suffices to prove the claim for one fixed $v_0 \in V_{\iso}$. Indeed, then any $v \in V_{\iso}$ is of the form $f\cdot v_0$ for some $f \in \stab(\omega_7)$, and we get 
$$
    \Phi_v(\omega_7) = \Phi_{f \cdot v_0}(f\cdot \omega_7) = f \cdot \Phi_{v_0}(\omega_7) \in Gr_{\iso}(2,V_v), 
$$
since $f \cdot \omega:= \big(\bigwedge^p f\big)\big(\omega\big)$ and $\Phi_{f(v_0)} \circ \big(\bigwedge^3 f\big) = \big(\bigwedge^2 f\big) \circ \Phi_{v_0}$ for every $f \in SO(V)$.
So we take $v_0=e_{-3}$, and readily compute
\begin{equation*}
    \Phi_{e_{-3}}(\omega_7) = \bar{e}_{1}\wedge \bar{e}_2 \in \hat{Gr}_{\iso}(2,e_{-3}^{\perp}/{\langle e_{-3} \rangle}).
    \qedhere
\end{equation*}
\end{proof}

In summary, this shows how the \Cref{thm:mainTheorem} fails for $\hat{Gr}_{\iso}(3,7)$: by \Cref{claim:oddCounterExample} $\omega_7$ satisfies the assumption but is itself not in $\hat{Gr}_{\iso}(3,7)$. In particular, this means that $\hat{Gr}_{\iso}(3,7)$ cannot be defined by pulling back the equations of $\hat{Gr}_{\iso}(2,5)$ along IGCP maps of the form $\Phi_v$. We originally constructed our counterexample by analyzing where our proof fails if $\dim V = 7$. However, it turned out, that $\omega_7$ is interesting also from different points of view, which we will discuss in the following remarks. 

\begin{remark}
In $1900$, Engel \cite{Engel} showed that if $\omega$ is a generic $3$-form on $\CC^7$, its symmetry group is isomorphic to the exceptional group $G_2$, and that such a $3$-form gives rise to a bilinear form $\beta_{\omega}$. If we choose coordinates such that $\omega$ agrees with our form $\omega_7$, then this group $G_2$ is precisely the stabilizer $G$ we computed in \Cref{claim:oddCaseTransitive}, and $\beta_{\omega}$ is up to scaling equal to our bilinear form given by $J$. For more about $G_2$, we refer the reader to \cite{notesOnG2}.
\end{remark}

\begin{remark}
Alternatively we can construct $\omega_7$ as the triple product on the split octonions. Here we will follow the notation from \cite{Baez}.
Recall that the space $\HH$ of \emph{quaternions} is the 4-dimensional real vector space with basis $\{1,i,j,k\}$, equipped with a bilinear associative product specified by Hamilton's formula
$$
    i^2=j^2=k^2=ijk=-1.
$$
The \emph{conjugate} of a quaternion $x=a+bi+cj+dk$ is given by $\overline{x}=a-bi-cj-dk$. We also have a quadratic form given by $Q_{\HH}(x):=x\overline{x}=\overline{x}x=a^2+b^2+c^2+d^2$. 
The space of \emph{split octonions} is the vector space $\OO_s:= \HH \oplus \HH$ with a bilinear (but nonassociative) product given by
$$
    (a,b)(c,d):=(ac+d\bar{b},\bar{a}d+cb).
$$
The conjugate of an octonion $(a,b)$ is given by $\overline{(a,b)}=(\bar{a},-b)$, and we define a quadratic form $Q_{\OO_s}$, of signature $(4,4)$, by $Q_{\OO_s}(x)=x\overline{x}=\overline{x}x$; or equivalently $Q_{\OO_s}((a,b))=Q_{\HH}(a)-Q_{\HH}(b)$. We will write
\begin{align*}
    e_0 := (1,0) && e_1 := (i,0) && e_2 := (j,0) && e_3 := (k,0) \\
    e_4 := (0,1) && e_5 := (0,i) && e_6 := (0,j) && e_7 := (0,k).
\end{align*}
Let $\OO_{\text{Im}} = \{x \in \OO_s \mid \bar{x}=-x\} = \Span\{e_1, \ldots e_7\}$ denote the \emph{imaginary split octonions}. On $\OO_{\text{Im}}$ we can define a cross product given by the commutator:
$$
    x \times y := \frac{1}{2}(xy-yx),
$$
and a triple product $T: \OO_{\text{Im}} \times \OO_{\text{Im}} \times \OO_{\text{Im}} \to \RR$, given by
$$
    T(x,y,z):=\langle x, y \times z \rangle,
$$
where $\langle \cdot,\cdot \rangle$ is the bilinear form coming from $Q_{\OO_s}$. 
This triple product is an alternating trilinear form, and hence can be identified with an element $\omega \in \bigwedge^3V^*$, where $V=\OO_{\text{Im}}$. Explicitly, writing $e_i^* \in V^*$ for the dual vector to $e_i$, 
we have
\begin{align*}
    \omega =& e_1^*\wedge e_2^* \wedge e_3^* + e_1^*\wedge e_4^* \wedge e_5^* + e_1^*\wedge e_6^* \wedge e_7^* + e_2^*\wedge e_4^* \wedge e_6^*\\
    & - e_2^*\wedge e_5^* \wedge e_7^* + e_3^*\wedge e_4^* \wedge e_7^* + e_3^*\wedge e_5^* \wedge e_6^*.
\end{align*}
Note that the terms in $\omega$ correspond to the lines in the Fano plane:

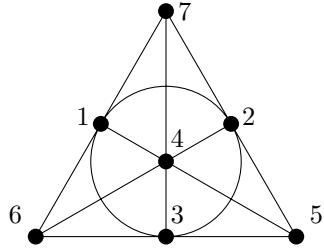
\begin{center}
\begin{tikzpicture}
  \draw (30:1)  -- (210:2)
        (150:1) -- (330:2)
        (270:1) -- (90:2)
        (90:2)  -- (210:2) -- (330:2) -- cycle
        (0:0)   circle (1);
  \fill (0:0)   circle(3pt)
        (30:1)  circle(3pt)
        (90:2)  circle(3pt)
        (150:1) circle(3pt)
        (210:2) circle(3pt)
        (270:1) circle(3pt)
        (330:2) circle(3pt);
    \node at (0.15,0.3) {$4$};
    \node at (1.1,0.6) {$2$};
    \node at (-1.1,0.6) {$1$};
    \node at (0.25,2) {$7$};
    \node at (0.15,-0.7) {$3$};
    \node at (-2,-0.7) {$6$};
    \node at (2,-0.7) {$5$};
\end{tikzpicture}

\end{center}
This $\omega$ agrees with $\omega_7$ up to a change of basis. Explicitly, if we substitute

\begin{align*}
    e_0 \mapsto \frac{e_4^*}{\sqrt{2}},&& e_1 \mapsto \frac{e_1^*+e_5^*}{\sqrt{2}},&& e_2 \mapsto \frac{e_2^*+e_6^*}{\sqrt{2}},&& e_3 \mapsto \frac{e_3^*+e_7^*}{\sqrt{2}}, \\
    &&e_{-1} \mapsto \frac{e_1^*-e_5^*}{\sqrt{2}},&& e_{-2} \mapsto \frac{e_2^*-e_6^*}{\sqrt{2}},&& e_{-3} \mapsto \frac{e_3^*-e_7^*}{\sqrt{2}}
\end{align*}
into $\omega_7$, we recover $\omega$ (up to scaling).
\end{remark}

\subsection{Counterexample in dimension 8}
Let $V$ be an $8$-dimensional vector space with basis $\{e_1,e_2,e_3,e_4,e_{-1},e_{-2},e_{-3},e_{-4}\}$, and quadratic form given by the matrix
$$
    J=\begin{pmatrix}
     0 & \rvline & I_4 \\
    \hline
     I_4 & \rvline & 0 \\
    \end{pmatrix}.
$$
Choose
\begin{align}
	\omega_8 := & 2e_1\wedge e_2 \wedge e_3 \wedge e_4 + 2e_{-1} \wedge e_{-2} \wedge e_{-3} \wedge e_{-4} \nonumber\\
	&+ e_1 \wedge e_2 \wedge e_{-1} \wedge e_{-2} + e_1 \wedge e_3 \wedge e_{-1} \wedge e_{-3} + e_1 \wedge e_4 \wedge e_{-1} \wedge e_{-4} \\
	&+ e_2 \wedge e_3 \wedge e_{-2} \wedge e_{-3} + e_2 \wedge e_4 \wedge e_{-2} \wedge e_{-4} + e_3 \wedge e_4 \wedge e_{-3} \wedge e_{-4}. \nonumber
\end{align}
One can verify that $\omega_8 \notin \hat{Gr}(4,V)$, so in particular $\omega_8 \notin \hat{Gr}_{\iso}(4,V)$. As before, we consider the algebraic group
\begin{align*}
SO(V) = & \{\phi \in SL(V) \mid \langle \phi(x),\phi(y) \rangle = \langle x, y \rangle \quad \forall x,y \in V\} \\ = & \{A \in SL(8,\KK) \mid A^TJA=J\}
\end{align*}
and its subgroup
$$
    G:=\stab(\omega_8)=\{\phi \in SO(V) \mid \phi \cdot \omega_8 = \omega_8 \}.
$$
\begin{claim}
The action of $G$ on $V_{\iso}$ is transitive.
\end{claim}

\begin{proof}
Take any $v_0 \in V_{\iso}$; we want to show that its orbit $G \cdot v_0$ has dimension equal to $\dim V_{\iso} = 7$. The Lie algebra of $SO(V)$ is given by
\begin{align*}
    \mathfrak{so}(V) =& \{X \in \mathfrak{sl}(8,\KK) \mid X^{T}J+JX=0\} \\ =& \Bigg\{\begin{pmatrix}[1.2]
     a & \rvline & b \\
     \hline
    c & \rvline & -a^T \\
    \end{pmatrix} a,b,c \in \KK^{4 \times 4}, b+b^T=c+c^T=0\Bigg\}.
\end{align*}
We introduce the following notation 
$$
    \text{for }b=\begin{pmatrix}
    0 & b_{12} & b_{13} & b_{14} \\ -b_{12} & 0 & b_{23} & b_{24} \\ -b_{13} & -b_{23} & 0 & b_{34} \\ -b_{14} & -b_{24} & -b_{34} & 0
    \end{pmatrix}
    \text{ write }
    \tilde{b} := \begin{pmatrix}
    0 & -b_{34} & b_{24} & -b_{23} \\ b_{34} & 0 & -b_{14} & b_{13} \\ -b_{24} & b_{14} & 0 & -b_{12} \\ b_{23} & -b_{13} & b_{12} & 0
    \end{pmatrix}.
$$
We can compute the Lie algebra $\mathfrak{g} \subset \mathfrak{so}(V)$ of $G$ as follows:
\begin{align*}
    \mathfrak{g} =& \{X \in \mathfrak{so}(V) \mid X\cdot \omega_8 =0\} \\ =& \Bigg\{\begin{pmatrix}[1.4]
    a & \rvline & b \\
     \hline
     \tilde{b} & \rvline & -a^T \\
    \end{pmatrix} \mid a \in \mathfrak{sl}(4,\KK), b+b^T=0\Bigg\}.
\end{align*}
As before $\dim G =\dim \mathfrak{g}=21$. For the stabilizer, if we take $v_0=e_{-4} \in V_{\iso}$, we see that
\begin{align*}
    \stab_{\mathfrak{g}}(v_0) =& \{X \in \mathfrak{g} \mid X\cdot e_{-4} =0\}
\end{align*}
is the set of matrices in $\mathfrak{g}$ whose final column is zero, which has dimension $14$. So $\dim(G \cdot v_0) = 21 - 14 = 7 = \dim V_{\iso}$, as desired.
\end{proof}
As before, we conclude the following claim.
\begin{claim}
    For every $v \in V_{\iso}$, it holds that $\Phi_v(\omega_8) \in \hat{Gr}_{\iso}(3,V_v)$.
\end{claim}

\begin{proof}
As in \Cref{claim:oddCounterExample}, it suffices to prove the claim for one fixed $v \in V_{\iso}$. Taking $v=e_{-1}$, we compute that \begin{equation*}
    \Phi_{v}(\omega_8) = 2\bar{e}_2 \wedge \bar{e}_3 \wedge \bar{e}_4 \in \hat{Gr}_{\iso}(3,e_{-1}^{\perp}/{\langle e_{-1} \rangle}). \qedhere
\end{equation*}
\end{proof}

In summary, this shows how the \Cref{thm:mainTheorem} fails for $\hat{Gr}_{\iso}(4,8)$. As before, this means that $\hat{Gr}_{\iso}(4,8)$ cannot be defined by pulling back the equations of $\hat{Gr}_{\iso}(3,6)$ along IGCP maps of the form $\Phi_v$.

\begin{remark}
The Lie algebra $\mathfrak{g}$ defined above is in fact isomorphic to $\mathfrak{so}(7)$. An explicit isomorphism $\mathfrak{so}(7) \to \mathfrak{g}$ can be given by

\vspace{1em}
\noindent
\resizebox{\textwidth}{!} 
{
$
\begin{pmatrix}
    0 & \rvline & -2y_1 & -2y_2 & -2y_3 & \rvline & -2x_1 & -2x_2 & -2x_3 \\
    \hline 
    x_1 & \rvline & a_{11} & a_{12} & a_{13} & \rvline & 0 & b_{12} & b_{13} \\
    x_2 & \rvline & a_{21} & a_{22} & a_{23} & \rvline & -b_{12} & 0 & b_{23} \\
    x_3 & \rvline & a_{31} & a_{32} & a_{33} & \rvline & -b_{13} & -b_{23} & 0 \\
    \hline 
    y_1 & \rvline & 0 & c_{12} & c_{13} & \rvline & -a_{11} & -a_{21} & -a_{31} \\
    y_2 & \rvline & -c_{12} & 0 & c_{23} & \rvline & -a_{12} & -a_{22} & -a_{32} \\
    y_3 & \rvline & -c_{13} & -c_{23} & 0 & \rvline & -a_{13} & -a_{23} & -a_{33} \\
\end{pmatrix} 
\mapsto 
\begin{pmatrix}
    d_{11} & a_{12} & a_{13} & -c_{23} & \rvline & 0 & -y_3 & y_2 & x_1 \\
    a_{21} & d_{22} & a_{23} & c_{13} & \rvline & y_3 & 0 & -y_1 & x_2 \\
    a_{31} & a_{32} & d_{33} & -c_{12} & \rvline & -y_2 & y_1 & 0 & x_3 \\
    b_{23} & -b_{13} & b_{12} & d_{44} & \rvline & -x_1 & -x_2 & -x_3 & 0 \\
    \hline 
    0 & -x_3 & x_2 & y_1 & \rvline & -d_{11} & -a_{21} & -a_{31} & -b_{23} \\
    x_3 & 0 & -x_1 & y_2 & \rvline & -a_{12} & -d_{22} & -a_{32} & b_{13} \\
    -x_2 & x_1 & 0 & y_3 & \rvline & -a_{13} & -a_{23} & -d_{33} & -b_{12} \\
    -y_1 & -y_2 & -y_3 & 0 & \rvline & c_{23} & -c_{13} & c_{12} & -d_{44} \\
\end{pmatrix}
$,
}
\vspace{1em}

\noindent
where for the left hand side we used the notation from \Cref{sec:counterexampleOdd}, and in the right hand side we have
\vspace{1em}

\noindent
\resizebox{\textwidth}{!} 
{
  $d_{11}:=\dfrac{a_{11}-a_{22}-a_{33}}{2}$, $d_{22}:=\dfrac{-a_{11}+a_{22}-a_{33}}{2}$, $d_{33}:=\dfrac{-a_{11}-a_{22}+a_{33}}{2}$, $d_{44}:=\dfrac{a_{11}+a_{22}+a_{33}}{2}$.
}

\end{remark}
\appendix

\section{Ranks of defining quadrics} \label{appendix:RankOfQuadrics}

In Section \ref{sec:comp}, we will finish the proof of \Cref{rank4}, by verifying the following fact:
\begin{claim} \label{claim:rank4Small}
    $\hat{Gr}_{\iso}(3,7)$, as well as both irreducible components of $\hat{Gr}_{\iso}(4,8)$, can be set-theoretically defined by quadrics of rank at most $4$.
\end{claim}
In Section \ref{sec:Cartan}, we explain how \Cref{rank4} follows from the literature on isotropic Grassmannians, in particular the Cartan embedding.

\subsection{Computational approach} \label{sec:comp}
Our verification is based on an algorithm, which we implemented in Macaulay2 \cite{M2}. We sketch the steps of the algorithm below. Let $X$ be either $\hat{Gr}_{\iso}(3,7)$, or one of the components of $\hat{Gr}_{\iso}(4,8)$.
\begin{enumerate}
    \item Compute the ideal $I$ defining $X$ by parametrizing an open subset and performing a Gr\"obner basis computation. The ideal $I$ is generated by linear equations and quadrics.
    \item Get rid of the linear equations by substituting variables.
    \item View the space $I_2$ of quadrics in $I$ as a representation of $SO(V)$, and decompose it into weight spaces.
    \item Find a highest weight vector $p \in I_2$ of minimal rank.
    \item\label{it:AlgoStep} Compute the subrepresentation generated by $p$, using the lowering operators in $\mathfrak{so}(V)$. 
    \item If we generated all of $I_2$, we are done.
    \item Otherwise, find the highest weight space we did not yet generate, let $p$ be a quadric of minimal rank in it, and return to step (\ref{it:AlgoStep}).
\end{enumerate}
By construction, the $SO(V)$-orbits of the quadrics $p$ we found give sufficiently many equations to define $X$. Since acting with $SO(V)$ does not change the rank of a quadric, it follows that if each of our quadrics has rank at most $4$, then $X$ can be defined by quadrics of rank at most $4$.
For $\hat{Gr}_{\iso}(3,7)$, our algorithm returned the following quadrics:
\begin{align*}
& x_{0,1,2}^2+2x_{1, 2, 3}x_{1, 2, -3}, \\
& x_{0, 1, 2}(x_{1, 2, -2} + x_{1, 3, -3})+2x_{1, 2, 3}x_{0, 1, -3}, \\ 
& x_{0, 1, 3}x_{0, 1, -3}+x_{0, 1, 2}x_{0, 1, -2}, \\
& x_{1,2,3}(x_{0,1,-1}+x_{0,2,-2}-x_{0,3,-3})+x_{0,1,2}(x_{2,3,-2}+x_{1,3,-1}),\\
& x_{0,1,-1}^2-(x_{0,2,-2}+x_{0,3,-3})^2+2(x_{1,3,-3}+x_{1,2,-2})(x_{3,-1,-3}+x_{2,-1,-2}).
\end{align*}
For one of the components of $\hat{Gr}_{\iso}(4,8)$, we found the following quadrics:
\begin{align*}
& x_{1,2,3,-3}^2-x_{1,2,3,4}x_{1,2,-3,-4}, \\
& 2x_{1,2,3,-3}x_{1,3,4,-1} - x_{1,2,3,4}(x_{1,2,-1,-2}-x_{1,3,-1,-3}-x_{1,4,-1,-4}), \\
& (x_{1,4,-1,-4}+x_{2,4,-2,-4}-x_{3,4,-3,-4})^2-4x_{3,4,-1,-2}x_{1,2,-3,-4}.
\end{align*}
Since all quadrics listed above have rank at most $4$, and since both components of $\hat{Gr}_{\iso}(4,8)$ are isomorphic, our verification is now complete.

\subsection{Rank 4 quadrics via the Cartan embedding} \label{sec:Cartan}

In this section we will sketch an alternative proof that $Gr_{\iso}(p,2p+1)$ and the connected components of $Gr_{\iso}(p,2p)$, in their Pl\"ucker embedding, are defined by linear equations and quadrics of rank at most $4$, using the Cartan embedding (sometimes called spinor embedding), cf. \cite{Cartan}  or \cite[Appendix E]{harnad2021tau}.
The proof follows by combining the following facts:
\begin{itemize}
    \item The image of the Cartan embedding is defined by quadrics \cite{Cartan}. 
    \item The Pl\"ucker embedding factors as the Cartan embedding followed by a degree two Veronese embedding (\cite[Theorem 2.1]{Balogh} and \cite[Theorem 1]{Pasini}). 
    \item The image of a degree two Veronese embedding is defined by quadrics of rank 3 and 4.
\end{itemize}
The idea is that the Veronese embedding turns the quadratic equations of the Cartan embedding into linear equations, so the only quadrics we need are the ones coming from the Veronese embedding.
\bibliography{bibliography}{}
\bibliographystyle{alpha}
\end{document}